\documentclass[a4paper,11pt]{article}
\usepackage{amsmath, amssymb, amsthm}
\usepackage{graphicx} 
\usepackage[all]{xy}
\usepackage{xcolor}
\usepackage[latin1]{inputenc}

\graphicspath{{./dibujos/}}

\theoremstyle{definition}
\newtheorem{theorem}{Theorem}[section]
\newtheorem{proposition}[theorem]{Proposition}
\newtheorem{lemma}[theorem]{Lemma}

\newtheorem{definition}[theorem]{Definition}

\newtheorem{remark}[theorem]{Remark}
\newtheorem{corollary}[theorem]{Corollary}
\newtheorem{claim}[theorem]{Claim}

\newcommand{\R}{\mathbf{R}}  
\newcommand{\D}{\mathbf{D}}  
\newcommand{\N}{\mathbf{N}}  

\usepackage{hyperref}
\hypersetup{
    bookmarks=true,         
    unicode=false,          
    pdftoolbar=true,        
    pdfmenubar=true,        
    pdffitwindow=false,     
    pdfstartview={FitH},    
    pdftitle={My title},    
    pdfauthor={Author},     
    pdfsubject={Subject},   
    pdfcreator={Creator},   
    pdfproducer={Producer}, 
    pdfkeywords={keyword1} {key2} {key3}, 
    pdfnewwindow=true,      
    colorlinks=false,       
    linkcolor=red,          
    citecolor=green,        
    filecolor=magenta,      
    urlcolor=cyan           
}

\begin{document}

\def\SSS{\mathfrak S} \def\VB{{\rm VB}} \def\KB{{\rm KB}}
\def\VGamma{{\rm V}\Gamma} \def\SS{\mathcal S} \def\XX{\mathcal X}
\def\Aut{{\rm Aut}} \def\Id{{\rm Id}} \def\YY{\mathcal Y}
\def\I{{\rm I}} \def\II{{\rm II}} \def\III{{\rm III}}
\def\op{{\rm op}}


\title{\bf{Virtual braids from a topological viewpoint}}

\author{\textsc{Bruno A. Cisneros de la Cruz}}

\date{ }

\maketitle

\begin{abstract}
\noindent
  Virtual braids are a combinatorial generalization of braids. We present abstract braids as equivalence classes of braid diagrams on a surface, joining two distinguished boundary components. They are identified up to isotopy, compatibility, stability and Reidemeister moves. We show that virtual braids are in a bijective correspondence with abstract braids. Finally we demonstrate that for any abstract braid, its representative of minimal genus is unique up to compatibility and Reidemeister moves. The genus of such a representative is thus an invariant for virtual braids. We also give a complete proof of the fact that there is a bijective correspondence between virtually equivalent virtual braid diagrams and braid-Gauss diagrams.
\end{abstract}

\tableofcontents



\section{Introduction}

  A knot diagram is an oriented planar closed curve in general position (only transversal double points, called crossings) with a function that assigns to each crossing a sign, either positive or negative. Knot diagrams are identified up to Reidemeister moves and the equivalence classes are in bijective correspondence with knots in $\R ^3$. Some approaches to knot theory are by means of knot diagrams.

  There is a combinatorial way to describe an oriented planar closed curve in general position called the Gauss word. It was described by Gauss \cite[pp. 85]{Ca95} in his unpublished notebooks. The main idea of this approach is to consider the curve as an oriented graph where the vertices are the crossings of the curve and the edges are the oriented segments joining two crossings. 
  
  This idea has been retaken by O. Viro and M. Polyak in order to express knot diagrams in a combinatorial way. To do this, they introduced the notion of Gauss diagrams to compute Vassiliev's invariants \cite{PV94}. In fact  Mikhail N. Goussarov proved that all Vassiliev invariants can be calculated in this way. 
  
  The theory of virtual knots was introduced by L. Kauffman \cite{Ka97} as a generalization of classical knot theory. A virtual knot diagram is an oriented planar closed curve in general position with a function that assigns to each crossing a value that can be positive, negative or virtual. Virtual knot diagrams are identified up to Reidemeister, virtual and mixed moves. 
 
  Goussarov, Polyak and Viro \cite{GPV2000} showed that there is a bijection between virtually equivalent virtual knot diagrams and Gauss diagrams. Moreover they showed that any Vassiliev invariant of classical knots can be extended to virtual knots and calculated via Gauss diagrams formulas.

  Braids are a fundamental part of knot theory, as each link can be represented as the closure of a braid \cite{A23} (Alexander's theorem) and there is a complete characterization of the closure of braids given by Markov's theorem \cite{M35}. 
  
  L. Kauffman defined virtual braids and virtual string links and he also gave a virtual version of Alexander's theorem \cite{Ka2006}. Independendtly S. Kamada also proved a virtual version of Alexander's theorem and a full characterization of the closure of virtual braids, i.e. a virtual version of Markov's theorem \cite{K2007}. 
  
  Goussarov, Polyak and Viro defined Gauss diagrams for virtual string links. All though they stated that, up to virtual and mixed moves, each Gauss diagram defines a unique virtual string link diagram \cite{GPV2000,KP2010}, it is not clear that this statement is still true for virtual braids. 
  
  In Section 2 we describe the braid version of Gauss diagrams and then we prove that each braid-Gauss diagram defines, up to virtual and mixed moves, a unique virtual braid diagram. Then we introduce the $\Omega$ moves in braid-Gauss diagrams and we show that there is a bijection a bijection between the $\Omega$ equivalence classes of braid-Gauss diagrams and the virtual braids. Finally we recover the presentation given for the pure virtual braids in \cite{Ba2004}. 
 
  This result is quite technical and it has been considered as folklore in literature, even though it was missing a rigorous proof.  Gauss diagrams allow us to manage global information with much more liberty as they express the interaction among all strands along the time. On the other hand, on Gauss digrams we can obvious the virtual crossings, as the valuable information of virtual objects underlies on how the strands interacts with themselves through regular crossings, this information is expressed on the arrows of the Gauss diagrams. In particular they become our main tool to prove the results of Section 3 and in \cite{KP2010} are used to define and calculate Milnor invariants of virtual string links. 

  On the other hand, classical knot theory works with topological objects that can be studied with topological, analytic, algebraic and combinatorial tools. Virtual knots diagrams encodes the combinatorial information of a Gauss diagram, but these are not topological objects. A topological interpretation of these objects was done by N. Kamada, S. Kamada and J. Carter \cite{CKK2002, KK2000}. They defined abstract links as link diagrams on surfaces, identified up to stable equivalence and Reidemeister moves. They proved that abstract links are in bijective correspondence with virtual links.
  
   A representation of a virtual knot in a closed surface is called a {\it realization} of the virtual knot. The stable equivalence identifies different realizations, which means that a virtual knot may be realized in different surfaces. G. Kuperberg proved that any virtual link admits a realization in a surface of minimal genus and, up to diffeomorphism and Reidemeister moves, this realization is unique \cite{Ku2003}.
  


     In this paper we provide a topological interpretation of virtual braids inspired by \cite{KK2000, Ku2003}. In Section 3 we introduce the notion of abstract braid diagrams, that are braid diagrams in a surface with two distinguished boundary components and a real smooth function satisfying some conditions.    We also introduce the stable equivalence of abstract braid diagrams. The {\it abstract braids} are the abstract braid diagrams identified up to compatibility, stable equivalence and Reidemeister moves.  We prove that abstract braids are in bijective correspondence with virtual braids. 
  
     The notion of abstract braid diagram must not be confused with the definition given in \cite{KK2000}, even if the concept is quite similar. An abstract diagram in \cite{KK2000} is a pair $(S,D)$, with $S$ an oriented, compact surface and $D$ is a diagram in $S$ such that $D$ is a deformation retract of $S$. In any case, what we define as abstract braid diagram corresponds to the realization of an abstract diagram in \cite{KK2000}.

    In Section 4    (Theorem \ref{thm:3})  we prove that given an abstract braid, there is a unique abstract braid diagram (up to Reidemeister moves and compatibility) of minimal genus.

	These results states some questions for future work:
	\begin{enumerate}
	\item For any virtual braid diagram there exists a minimal thickened abstract braid representative. Can this representative induce a normal form on virtual braids? 
	\item We can see virtual braids as virtual string links, but in virtual string links we have more Reidemeister and virtual moves. Thus, given two virtual braids equivalent as virtual string link, are they equivalent as virtual braids?  i.e. Does virtual braids embeds in virtual string links? 
	\item Given a thickened abstract braid diagram $\bar{\beta}=(M_S,f,\beta)$, what is the relation between the fundamental group of $M_S\setminus \beta$ and the group of the virtual link \cite{BB2014,KK2000} obtained by the closure of $\bar{\beta}$?
	\end{enumerate}

\section{Virtual braids and Gauss diagrams}

We fix the next notation: set $n$ a natural number, the interval $[0,1]$ is denoted by $I$, and the $2$-cube is denoted by $\D = I \times I$.  The projections on the first and second coordinate from the $2$-cube to the interval, are denoted by $\pi_1 : \D \rightarrow I$ and $\pi_2: \D \rightarrow I$, respectively. A set of planar curves is said to be in {\it general position} if all its multiple points are transversal double points.

\subsection{Virtual braids.}

\begin{definition}\label{def:vbd}
 A {\it strand diagram} on $n$ strands is an $n$-tuple of curves, $\beta=(\beta_1,\dots,\beta_n)$, where $\beta_k : I \rightarrow \D$ for $k=1,\dots,n$, such that:
 \begin{enumerate}
  \item There exists $\sigma \in S_n$ such that, for $k=1,\dots,n$, we have $\beta_k(0) = a_k$ and $\beta_k(1) = b_{\sigma(k)}$, where $a_k = (0, \frac{k}{n+1})$ and $b_k = (1, \frac{k}{n+1})$.
  \item For $k=1,\dots,n$ and all $t\in I$, $(\pi_1 \circ \beta_k)(t) = t$.
  \item The set of curves in $\beta$ is in general position. 
 \end{enumerate}
 The curves $\beta_k$ are called {\it strands} and the transversal double points are called {\it crossings}. The set of crossings is denoted by $\mathcal{C}(\beta)$. 
 
 A {\it virtual braid diagram} on $n$ strands is a strand diagram on $n$ strands $\beta$ endowed with a function $\epsilon: \mathcal{C}(\beta) \rightarrow \{+1,-1,v\}$. The crossings are called {\it positive, negative} or {\it virtual} according to the value of the function $\epsilon$. The positive and negative crossings are called {\it regular crossings} and the set of regular crossings is denoted by $R(\beta)$. In the image of a regular neighbourhood (homeomorphic to a disc sending the center to the crossing) we replace the image of the involved strands as in Figure \ref{pnvcross}, according to the crossing type.  
\end{definition}

\begin{figure} [h]\centering
\centering
 \includegraphics[scale=.5]{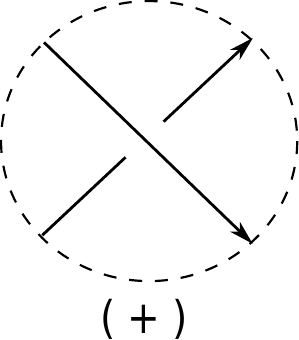}   \qquad \includegraphics[scale=.5]{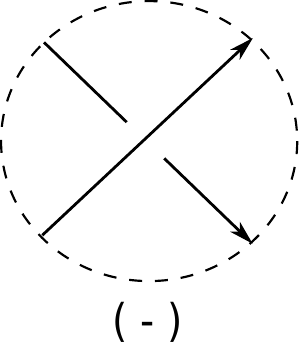} \qquad  \includegraphics[scale=.5]{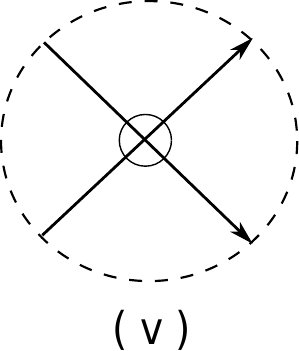}
	  \caption{Positive, negative and virtual crossings.}
	\label{pnvcross}
\end{figure}

Without loss of generality we draw the braid diagrams from left to right. We denote the set of virtual braid diagrams on $n$ strands by $VBD_n$.

\begin{definition}
 Given two virtual braid diagrams on $n$ strands, $\beta_1$ and $\beta_2$, and a neighbourhood $V\subset \D$, homeomorphic to a disc, such that:
 \begin{itemize}
 	\item Up to isotopy $\beta_1 \setminus V = \beta_2 \setminus V$.
 	\item Inside $V$, $\beta_1$ differs from $\beta_2$ by a diagram as either in Figure \ref{Rmoves}, or in Figure \ref{Vmoves}, or in Figure \ref{Mmoves}.
 \end{itemize} 
 Then we say that $\beta_2$ is obtained from $\beta_1$ by an $R2a$, $R2b$, $R3$, $V2$, $V3$, $M$ or $M'$ moves. 
 
 The moves $R2a$, $R2b$, $R3$ are called {\it Reidemeister moves}, the moves $V2$ and $V3$ are called {\it virtual moves}, and the moves $M$ and $M'$ are called {\it mixed moves}.

\begin{figure} [h] \centering
 \includegraphics[scale=0.5]{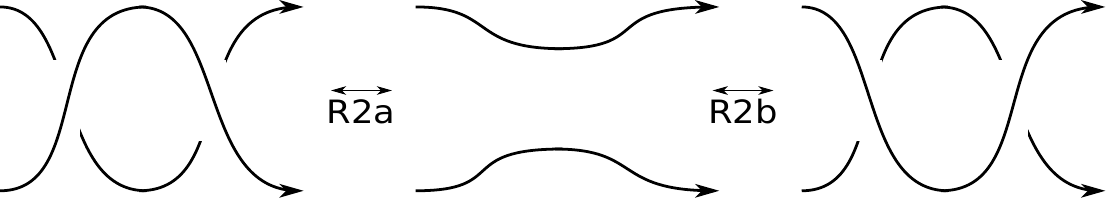} 
\quad 
 \includegraphics[scale=0.5]{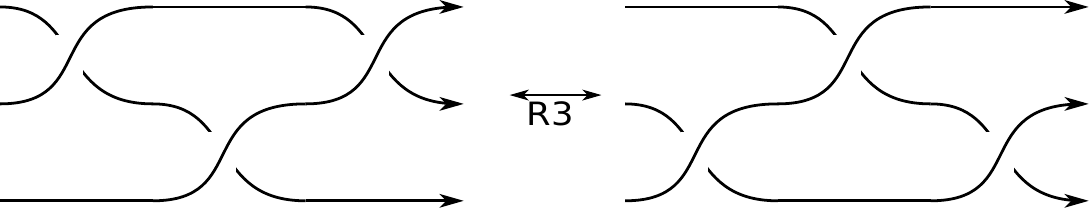} 
 \caption{Reidemester moves.}
 \label{Rmoves}
\end{figure}

\vspace{-8pt}
\begin{figure}  \centering
 \includegraphics[scale=0.5]{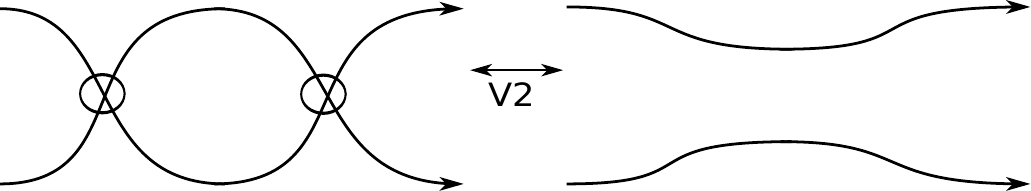}  
 \quad
 \includegraphics[scale=0.5]{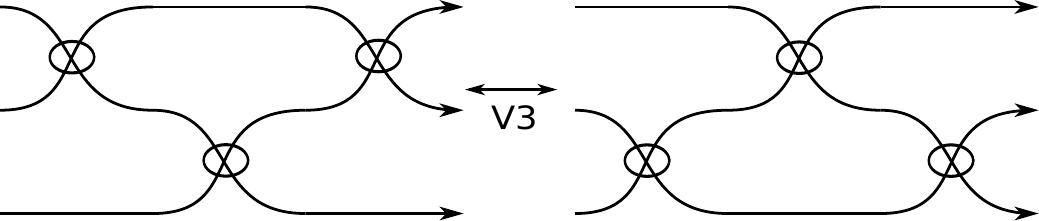}
 \caption{Virtual moves.}
 \label{Vmoves}
\end{figure}

\vspace{-8pt}
\begin{figure} \centering
 \includegraphics[scale=0.5]{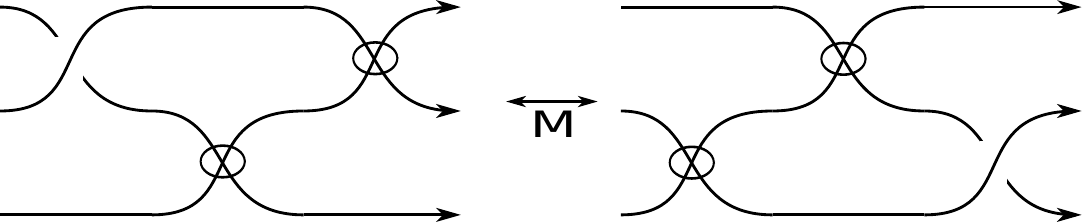} 
 \quad
 \includegraphics[scale=0.5]{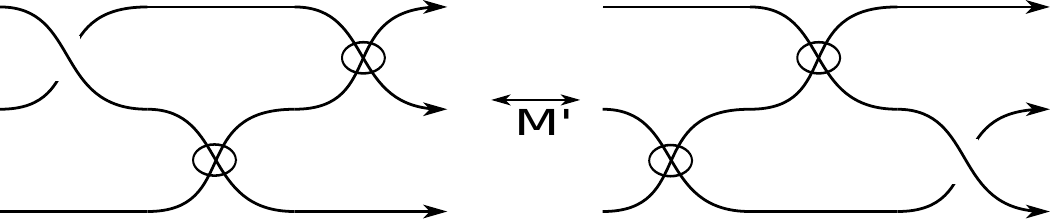}
 \caption{Mixed moves.}
 \label{Mmoves}
\end{figure}

Let $\beta$ and $\beta'$ be two virtual braid diagrams. Note that if $\beta$ can be obtained from $\beta'$ by a finite series of virtual, mixed or  Reidemeister moves, necessarily $\beta$ and $\beta '$ have the same number of strands. 

If $\beta'$ can be obtained from $\beta$ by isotopy and a finite number of virtual, Reidemeister or mixed moves, $\beta$ and $\beta'$ are {\it virtually Reidemeister equivalent}.  We denote this by $\beta\sim \beta'$.  These equivalence classes are called {\it virtual braids on $n$ strands}. We denote by $VB_n = VBD_n/\sim$ the set of virtual braids on $n$ strands.

If $\beta'$ can be obtained from $\beta$ by isotopy and a finite number of virtual or mixed moves, $\beta$ and $\beta'$ are {\it virtually equivalent}.  We denote this by $\beta\sim_{vm} \beta'$.

If $\beta'$ can be obtained from $\beta$ by isotopy and a finite number of Reidemeister moves, $\beta$ and $\beta'$ are {\it Reidemeister equivalent}.  We denote this by $\beta\sim_{R} \beta'$.

\end{definition}

\begin{remark}\label{rmk:VBn}
 Define the product of two virtual braids diagrams as the concatenation of the diagrams and an isotopy in the obtained diagram, to fix it in $\D$. With this operation the set of virtual braid diagrams has the structure of a monoid. It is not hard to see that it factorizes in a group when we consider the virtual Reidemeister equivalence classes. Thus, the set of virtual braids has the structure of a group with the product defined as the concatenation of virtual braids. The virtual braid group on $n$ strands has the following presentation:
\begin{itemize}
	\item Generators: $\sigma_1,\dots,\sigma_{n-1},\tau_1,\dots,\tau_{n-1}$.
	\item Relations: \[\begin{split}
					\sigma_i\sigma_{i+1}\sigma_i = \sigma_{i+1}\sigma_i\sigma_{i+1} \quad & \quad \text{for } 1\leq i \leq n-2 \\
					\sigma_i\sigma_j = \sigma_j\sigma_i 					\quad	& \quad\text{ if } |i-j|\geq 2 \\
					\tau_i \tau_{i+1}\tau_i = \tau_{i+1}\tau_i\tau_{i+1} 	\quad	& \quad\text{for } 1\leq i \leq n-2\\
					\tau_i\tau_j = \tau_j\tau_i 							\quad	& \quad\text{ if } |i-j|\geq 2 \\	
					\sigma_i \tau_{i+1}\tau_i = \tau_{i+1}\tau_i\sigma_{i+1} \quad 	& \quad\text{for } 1\leq i \leq n-2\\
					\tau_i\sigma_j = \sigma_j\tau_i 						\quad	& \quad\text{ if } |i-j|\geq 2 \\				
					\tau_i^2 = 1 \quad & \quad \text{for } 1\leq i \leq n-1 \\
				 \end{split}\]
\end{itemize}
\end{remark}

\begin{remark}
	The mixed moves can be replaced by the moves showed in Figure \ref{pmoves}.
\end{remark}

\begin{figure} [h]\centering
     \includegraphics[scale=0.5]{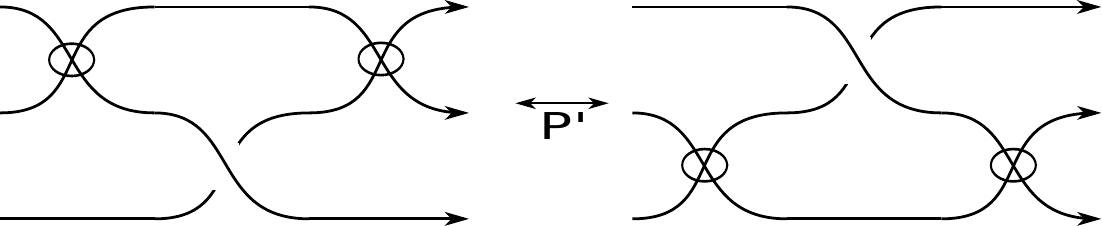} 
  \quad
   \includegraphics[scale=0.5]{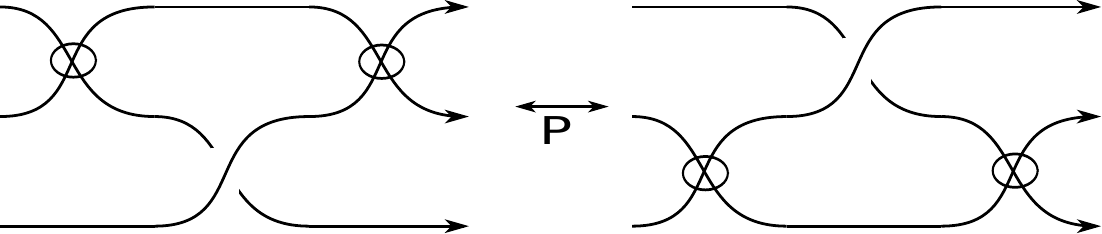}  
 \caption{Equivalent mixed moves.}
 \label{pmoves}
\end{figure}
 	
\subsection{Braid Gauss diagrams.}

\begin{definition}
  A {\it Gauss diagram} on $n$ strands $G$ is  an ordered collection of $n$ oriented intervals $\sqcup_{i=1}^{n}I_i$, together with a finite number of arrows and a permutation $\sigma \in S_n$ such that:
  \begin{itemize}
  	\item Each arrow connects by its ends two points in the interior of the intervals (possibly the same interval).
  	\item Each arrow is labelled with a sign $\pm 1$. 
  	\item The end point of the $i$-th interval is labelled with $\sigma(i)$. 
  \end{itemize}  
  Gauss diagrams are considered up to orientation preserving homeomorphism of the underlying intervals.
\end{definition}

\begin{figure} [h]\centering
	\includegraphics[scale=0.8]{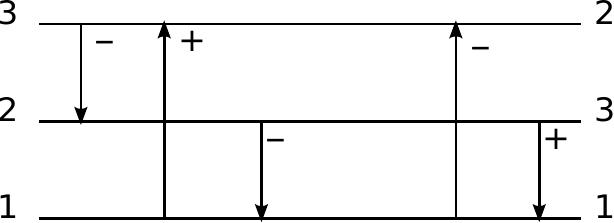}
	\quad
	\includegraphics[scale=0.8]{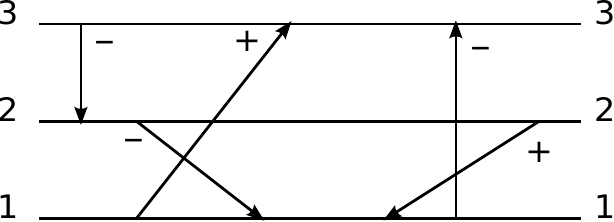}
	
	\caption{Gauss diagrams}
\end{figure}

\begin{definition}
  Let $\beta$ be a virtual braid diagram on $n$ strands. The {\it Gauss diagram} of $\beta$, $G(\beta)$, is a Gauss diagram on $n$ strands given by:
  \begin{itemize}
  	\item Each strand of $G(\beta)$ is associated to the corresponding strand of $\beta$. 
  	\item The endpoints of the arrows of $G(\beta)$ correspond to the preimages of the regular crossings of $\beta$. 
  	\item Arrows are pointing from the over-passing string to the under-passing string. 
  	\item The signs of the arrows are given by the signs of the crossings (their local writhe). 
  	\item The permutation of $G(\beta)$ correspond to the permutation associated to $\beta$.
  \end{itemize}

\end{definition}

\begin{figure} [h]\centering	
	\includegraphics[scale=0.7]{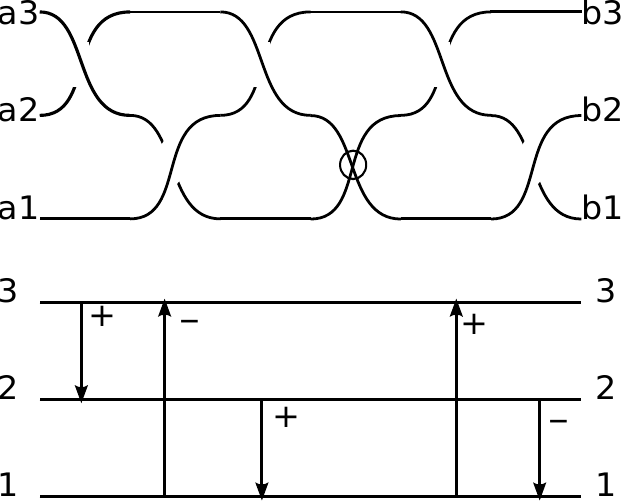}
	\quad
	\includegraphics[scale=0.7]{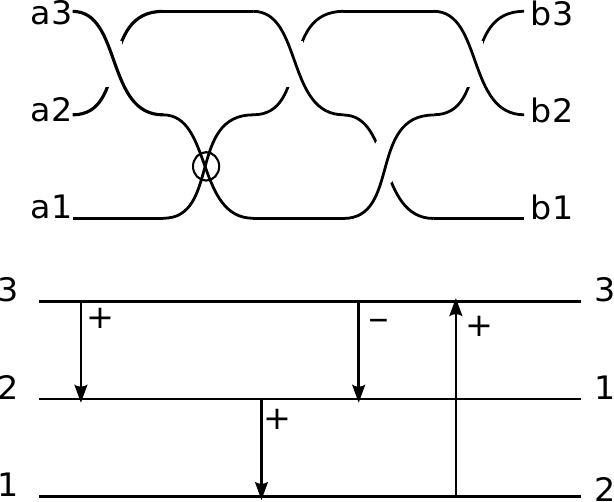}
	\caption{Gauss diagrams of virtual braid diagrams}
	\label{Fig:GdVB}
\end{figure}

 \begin{remark}\label{rmk:GDorto}
	The arrows of the Gauss diagram of any virtual braid diagram are pairwise disjoint and each arrow connects two different intervals. Furthermore we can draw them perpendicular to the underlying intervals, i.e. we can parametrize each interval $I_n$ with respect to the standard interval $I=[0,1]$, in such a way that the beginning and ending points of each arrow correspond to the same $t\in I$ and such that different arrows correspond to different $t$'s in $I$, see Figure \ref{Fig:GdVB}. 
\end{remark}

\begin{definition}
	Gauss diagrams satisfying the conditions of Remark \ref{rmk:GDorto} are called {\it braid Gauss diagrams}. The set of braid Gauss diagrams on $n$ strands is denoted by $bGD_n$.
\end{definition}

\begin{definition}\label{rmk:GaussOrder}
	Given a braid-Gauss diagram, $G$, we can associate a total order to the set of arrows in $G$, given by the order in which the arrows appear in the interval $I$, i.e. let $a$ and $b$ be two arrows in $G$, such that $a$ appears first, then $a>b$. This order is not defined in the equivalence class of the Gauss diagram, as it may change with orientation preserving homeomorphisms of the underlying intervals. 
	
	We denote by $P(G)$ the partial order obtained as the intersection of the total orders associated to $G$.  Given a virtual braid diagram $\beta$, let $G(\beta)$ be its Gauss diagram. Then $P(G(\beta))$ defines a partial order in the set of regular crossings,$R(\beta)$. We denote it by $P(\beta)$.
\end{definition}

\begin{theorem}\label{thm:1} 
    \begin{enumerate}
		\item Let $g$ be a braid-Gauss diagram on $n$ strands. Then there exists $\beta\in VBD_n$ such that $G(\beta)=g$.
		\item Let $\beta_1$ and $\beta_2$ be two virtual braids on $n$ strands. Then $G(\beta_1)=G(\beta_2)$ if and only if $\beta_1\sim_{vm} \beta_2$. 
	\end{enumerate}
\end{theorem}

From now on we fix $n\in \N$ the number of strands on the braid-Gauss diagrams, and we say braid-Gauss diagram instead of braid-Gauss diagram on $n$ strands. We split the proof of this theorem into some lemmas. 

\begin{lemma}\label{lemma:0}
	Let $g$ be a braid-Gauss diagram. Then there exists $\beta\in VBD_n$ such that $G(\beta)=g$.
\end{lemma}

\begin{proof}
	Let $g$ be a braid-Gauss diagram and $A=\{c_1,\dots,c_k\}$ be the set of arrows of $g$. Set a parametrization of the intervals as described in Remark \ref{rmk:GDorto}. This induces an order in $A$ given by  $c_i > c_j$ if $p_i < p_j$, where $p_i\in I$ is the corresponding endpoint of $c_i$. Suppose that  $c_i > c_j$ if $i<j$.
	
	Recall the notation of Definition \ref{def:vbd}. For $j=1,\dots,k$ let $d_j = (\frac{j}{k+1}, \frac{1}{2})$ and consider the disc $D_j$ with radius $r=\frac{1}{5(k+1)}$ centered in $d_j$. Draw a crossing inside $D_j$ according to the sign of $c_j$, and label the intersection of the crossing components with the boundary of $D_j$ as in Figure \ref{labelCross}.

\begin{figure}[h]\centering
	\includegraphics[scale=0.6]{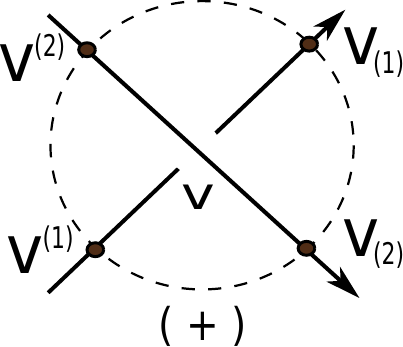}
	\qquad
	\includegraphics[scale=0.6]{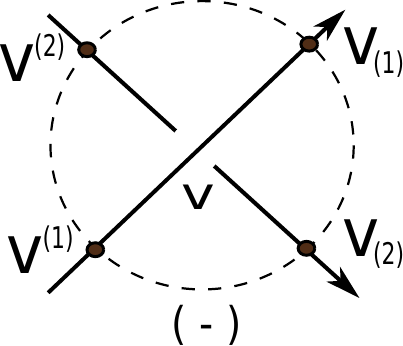}
	\caption{Labelled neighbourhoods of regular crossings.}
	\label{labelCross}
\end{figure}

	Drawing the strands: let $\sigma\in S_n$ be the permutation associated to $g$. Fix $i\in \{1,\dots,n\}$ and let $A_i=\{c_{i_1},\dots,c_{i_{m}}\}$ be the arrows starting or ending in the $i$-th interval.
	
%
%
	For $s=0,\dots,m$ define $o_s$ and $t_{s+1}$ as follows:
	\begin{enumerate}
		\item $o_0 = a_i$ and $t_{m+1}= b_{\sigma(i)}$.
		\item For $l=1,\dots,m$, $o_l = (d_{i_l})^{(v)}$ and $t_l = (d_{i_l})_{(v)}$ where:
			\begin{enumerate}
				\item If $c_{i_l}$ is a positive arrow starting in the $i$-th interval or a negative arrow ending in the $i$-th interval then $v=2$; 
				\item If $c_{i_l}$ is a negative arrow starting in the $i$-th interval or a positive arrow ending in the $i$-th interval then $v=1$. 
			\end{enumerate}
	\end{enumerate}	 
	
    For each $s\in \{0,\dots,m\}$, draw a curve joining $o_s$ to $t_{s+1}$ such that it is strictly increasing on the first component and disjoint from the discs $D_j$ for all $j\in \{1,\dots,k\}$ except possibly on the points $o_s$ and $t_{s+1}$ defined above. In this way we have drawn a curve joining $a_i$ with $b_\sigma(i)$ passing through the crossings $c_{i_1}, \dots, c_{j_m}$.
    
    For each $i\in \{1,\dots,k\}$ we can draw a curve as described before, so that they are in general position. Consider the double points outside the discs $D_j$ as virtual crossings. In this way we have constructed a virtual braid diagram such that its Gauss diagram coincides with $g$. 
\end{proof}

\begin{lemma}\label{lemma:01}
	Let $\beta_1$ and $\beta_2$ be two virtual braid diagrams on $n$ strands such that they are virtually equivalent. Then $G(\beta_1) = G(\beta_2)$. 
\end{lemma}

\begin{proof}
	In order to see this we only need to verify that the $V2$, $V3$, $M$ and $M'$ moves do not change the braid-Gauss diagram of a virtual braid diagram. In the cases of the $V2$ and $V3$ moves they involve only virtual crossings, which are not represented in the Gauss diagram, so they do not change the Gauss diagram. In the case of the $M$ and $M'$ moves, the Gauss diagrams of the equivalent virtual braid diagrams are equal (Figure \ref{GcM}), thus this type of move neither changes the Gauss diagram of the virtual braid diagram. 
\end{proof}

\begin{figure} [h]\centering
	
	\includegraphics[scale=0.8]{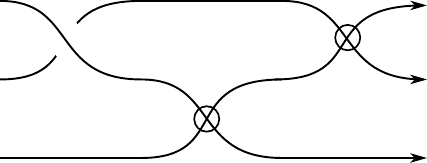} \quad
	\includegraphics[scale=0.8]{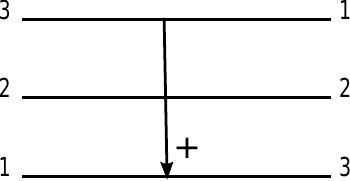} \quad
	\includegraphics[scale=0.8]{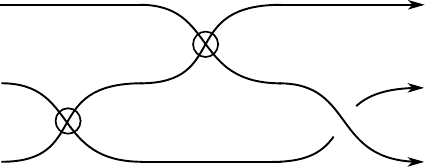}
	\caption{Gauss code of the mixed move.}
	\label{GcM}
\end{figure}

\begin{definition}\label{rmk:ordCross}
  Given $\beta \in VBD_n$ we can deform $\beta$ by an isotopy, in such a way that for $c_i, c_j\in \mathcal{C}(\beta)$  with, $i\neq j$, we have that $\pi_1(c_i) \neq \pi_1(c_j)$, in this case we say that $\beta$ is in \textit{general position}. If $\beta \in VBD_n$ is in general position, let $\mathcal{D}(\beta)$ be the total order associated to  $\mathcal{C}(\beta)$,  given by $c_i> c_j$ if $\pi_1(c_i) < \pi_1(c_j)$. Denote by $D(\beta)$ the total order of the set of regular crossings, $R(\beta)$, induced by $\mathcal{D}(\beta)$.
  
\end{definition}

\begin{definition}

A {\it primitive arc} of $\beta$ is a segment of a strand of $\beta$ which does not go through any regular crossing (but it may go through virtual ones). 

Let $\beta\in VBD_n$.  For $v\in R(\beta)$, set a disc $D_v$ centered in $v$, with a radius small enough so that its intersection with $\beta$ consists exactly in two transversal arcs as in Figure \ref{labelCross}. We denote by $v^{(1)}$ and by $v^{(2)}$ the bottom and upper left intersections of $\beta$ with $\partial D_v$, and by $v_{(2)}$ and by $v_{(1)}$ the bottom and upper right intersections of $\beta$ with $\partial D_v$ as in Figure \ref{labelCross}. 

 Let $d$ be a point in the diagram $\beta$, if $d\in \{b_1,\dots,b_n\}$ or $d\in \{c^{(1)}, c^{(2)}\}$ for some $c\in R(\beta)$ we denote $d$ by $d^*$. Similarly if $d\in \{a_1,\dots,a_n\}$ or $d\in \{c_{(1)}, c_{(2)}\}$ for some $c\in R(\beta)$ we denote $d$ by $d_*$.  A {\it joining arc} is a primitive arc $\alpha$ such that there exist $a_*$ and $b^*$ with $\alpha(0)=a_*$ and $\alpha(1)=b^*$.

\begin{figure}[h]\centering

	\includegraphics[scale=0.8]{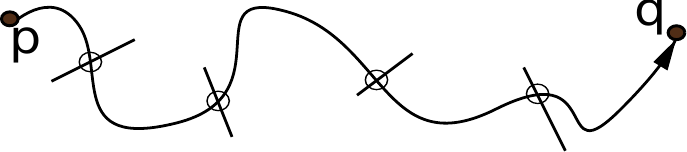} \qquad
	\includegraphics[scale=0.8]{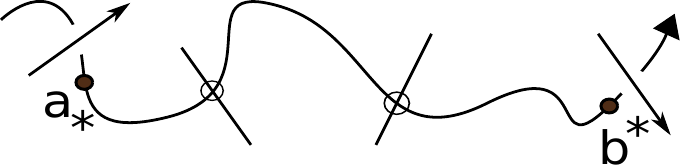} 
	\caption{primitive and joining arcs.} 
\end{figure} 
\end{definition}

As each arc is a segment of a strand $\beta_k: [0,1] \rightarrow \D $  we can parametrize it with respect to the projection on the first coordinate, i.e. there exists a continuous bijective map $\theta : [t_0,t_f] \rightarrow [0,1]$ with $0<t_0<t_f<1$ such that $\pi_1 (\alpha(\theta (t))) = t$ . Without loss of generality we suppose from now on that the arcs are parametrized by the projection on the first coordinate. 

\begin{lemma}\label{lemma:1}
	Let $\beta$ be a virtual braid diagram and let $\alpha_1, \alpha_2$ be two primitive arcs of $\beta$ such that: 
	\begin{enumerate}
		\item The arcs $\alpha_1$ and $\alpha_2$ start at the same time, $t_0$, and end at the same time, $t_f$.
		\item The arcs $\alpha_1$ and $\alpha_2$ start at the same point (a crossing which may be either virtual or regular), i.e. $\alpha_1(t_0)=\alpha_2(t_0)$.
		\item The arcs $\alpha_1$ and $\alpha_2$ do not intersect, except at the extremes, i.e.  $\alpha_1|_{(t_0,t_f)} \cap \alpha_2|_{(t_0,t_f)}  = \emptyset$.
	\end{enumerate}
	Then, there exists a virtual braid diagram $\beta'$ virtually equivalent to $\beta$ such that:
	\begin{enumerate}
		\item If $\beta_1$ and $\beta_2$ are the strands corresponding to $\alpha_1$ and $\alpha_2$ respectively, then up to isotopy they remain unchanged in $\beta'$, and in their restriction to  $(0,t_0)\times I$ we add only virtual crossings. 
		\item The diagrams $\beta$ and $\beta'$ coincide for $t\geq t_f$, i.e. $\beta|_{t\geq t_f} = \beta'|_{t\geq t_f}$.
		\item In $(t_0,t_f)\times I$ there are only virtual crossings with $\alpha_2$.
		\item If $\alpha_1(t_f) = \alpha_2(t_f)$, we can choose $\beta'$ such that there is no crossing in $(t_0,t_f)\times I$.
	\end{enumerate}
\end{lemma}

\begin{figure} [h]\centering
	\center
	\includegraphics[scale=0.6]{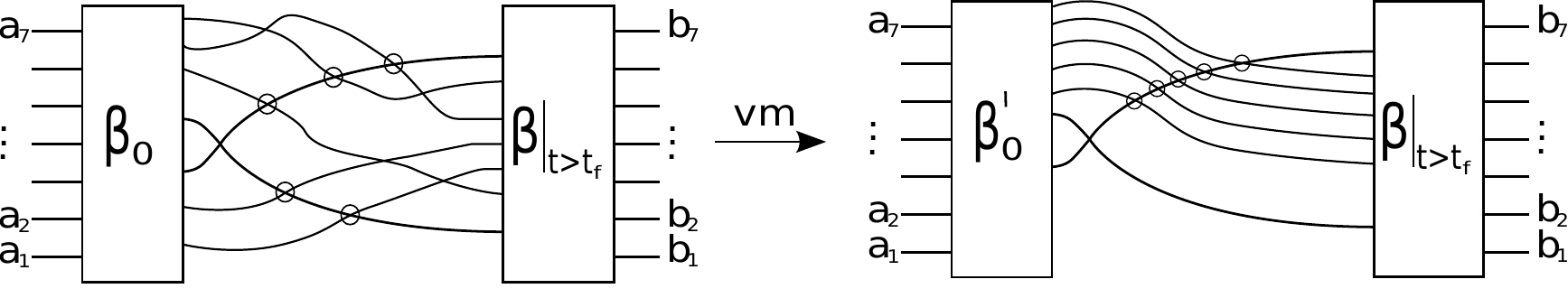} 
	\caption{Lemma \ref{lemma:1}.}
\end{figure}

\begin{proof}
  Suppose $\alpha_1(t_f)\neq \alpha_2(t_f)$ and that we have reduced $\beta$ by all the possible $V2$ moves that may be made on it. Note that $\alpha_1$, $\alpha_2$ and $y=t_f$ form a triangle $D$. 

  Let $C$ be the set of crossings in $\beta$ such that their projections on the first component are in the open interval $(t_0,t_f)$. Let $m_0$ be the number of crossings in $C$ that are in the interior of $D$, $m_1$ the number of crossings in $C$ that are on $\alpha_1$, $m_2$ the number of crossings in $C$ that are on $\alpha_2$, and $m_{\infty}$ the number of crossings in $C$ that are outside $D$. 

  We argue by induction on $m=m_0 + m_1 + m_{\infty}$. Suppose $m=1$. Then $C=\{c_1,\dots,c_{m_2},d\}$, where $c_1,\dots,c_{m_2}$ are the crossings on $\alpha_2$ and $d$ is the other crossing. We have four cases:
  \begin{enumerate}
    \item The crossing $d$ is outside $D$ (Figure \ref{lemaA1}). We move $d$ by an isotopy to the left part of $I\times [t_0,t_f)$.
     \begin{figure}[h]\centering
      \includegraphics[scale=0.6]{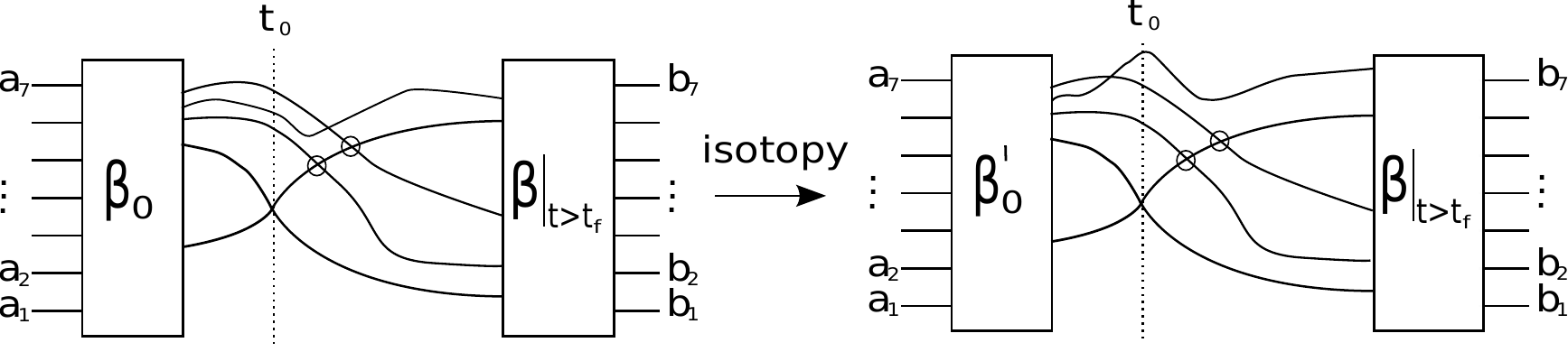}
      \caption{Case 1, Lemma \ref{lemma:1}.}
      \label{lemaA1}
    \end{figure}

    \item The crossing $d$ is inside $D$ (Figure \ref{lemaA2}). There are two strands entering $D$ that meet at the crossing $d$. We apply a move of type $M$, $M'$ or $V3$ according to the value of the crossing $d$, and then apply Case 1.
     \begin{figure}[h]\centering
      \includegraphics[scale=0.6]{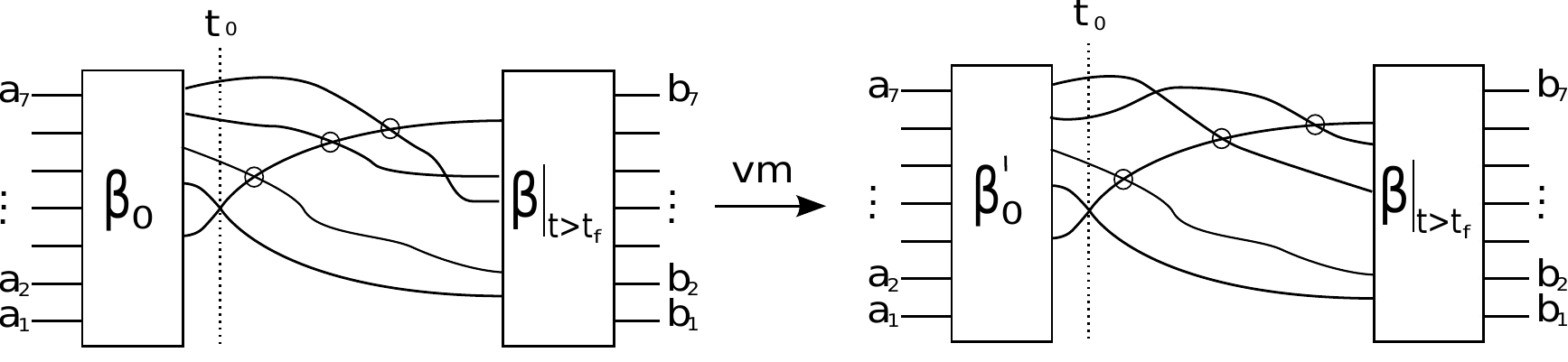}
      \caption{Case 2, Lemma \ref{lemma:1}.}
      \label{lemaA2}
    \end{figure}

    \item The crossing $d$ is on $\alpha_1$ in such a way that the strand making the crossing with $\alpha_1$ goes out $D$ (Figure \ref{lemaA3}). Then such strand also has a crossing with $\alpha_2$ and is the leftmost crossing on it. Up to isotopy we may apply a move of type $M$, $M'$ or $V3$ according to the value of the crossing $p$. Then we are done. 
     \begin{figure}[h]\centering
      \includegraphics[scale=0.6]{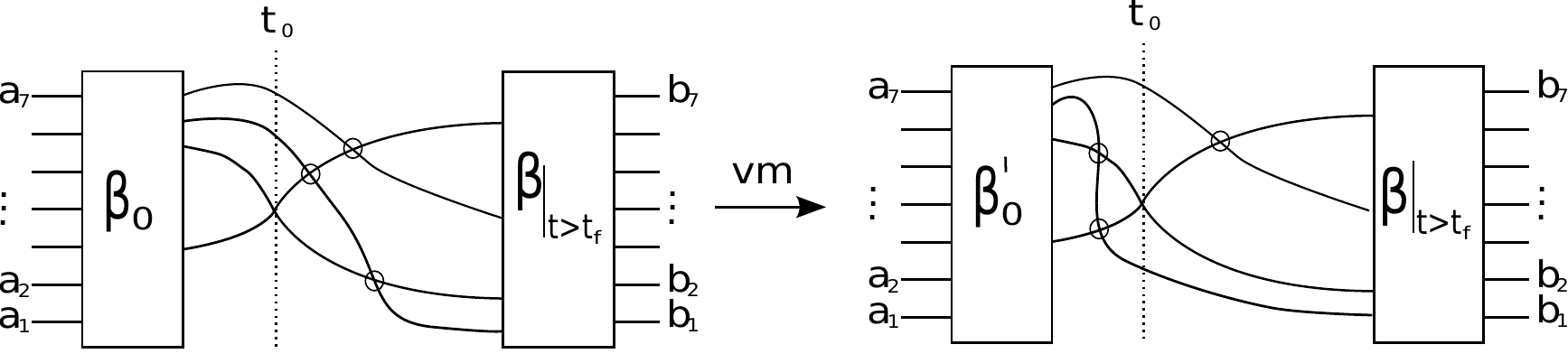}
      \caption{Case 3, Lemma \ref{lemma:1}.}
      \label{lemaA3}
    \end{figure}

    \item The crossing $d$ is on $\alpha_1$ in such a way that the strand making the crossing with $\alpha_1$ enters $D$ (Figure \ref{lemaA4}). Let $\beta_j$ be that strand. We apply a move of type $V2$ to $\beta_j$ and $\beta_2$ just before the crossing $p$, then we apply a move of type $P$, $P'$ or $V3$ according to the value of the crossing $p$. In this way now $d$ is a crossing on $\alpha_2$.
      \begin{figure}[h]\centering
	\includegraphics[scale=0.6]{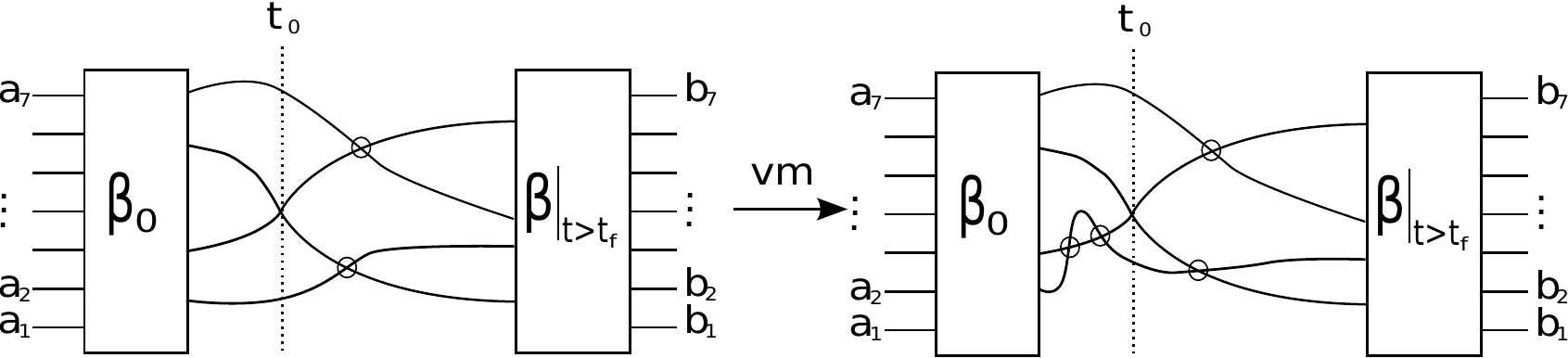} 
	
	\vspace{8pt}
	\includegraphics[scale=0.6]{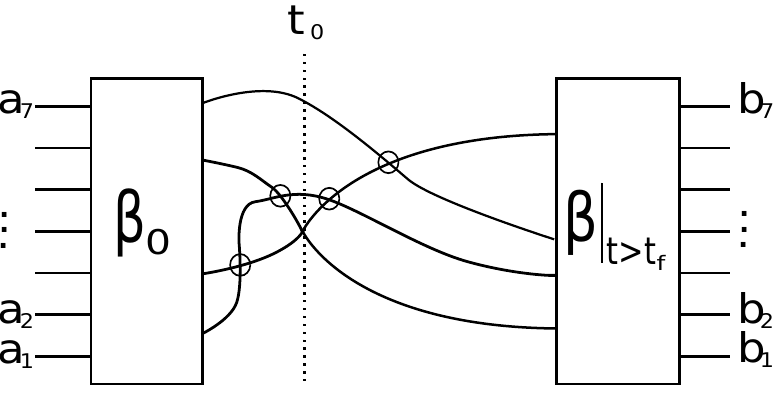} 
	\caption{Case 4, Lemma \ref{lemma:1}.}
	\label{lemaA4}
      \end{figure}
  \end{enumerate}

  Note that, in the above four cases, we have not deformed $\beta$ for $t\geq t_f$. Moreover, up to isotopy, $\beta_1$ and $\beta_2$ remain unchanged and in their restriction to $(0,t_0)\times I$ we have added only virtual crossings.  

  Now if $m\geq 2$ take $d$ the leftmost crossing in $(t_0,t_f) \times I$ such that $d$ is not on $\alpha_2$. We apply the case $m=1$ in order to get rid of this crossing and reduce the obtained diagram by all the possible $V2$ moves in it. By the induction hypothesis, we have proven (1,2,3) of the lemma.
  
  \begin{figure}[h]\centering
      \includegraphics[scale=0.6]{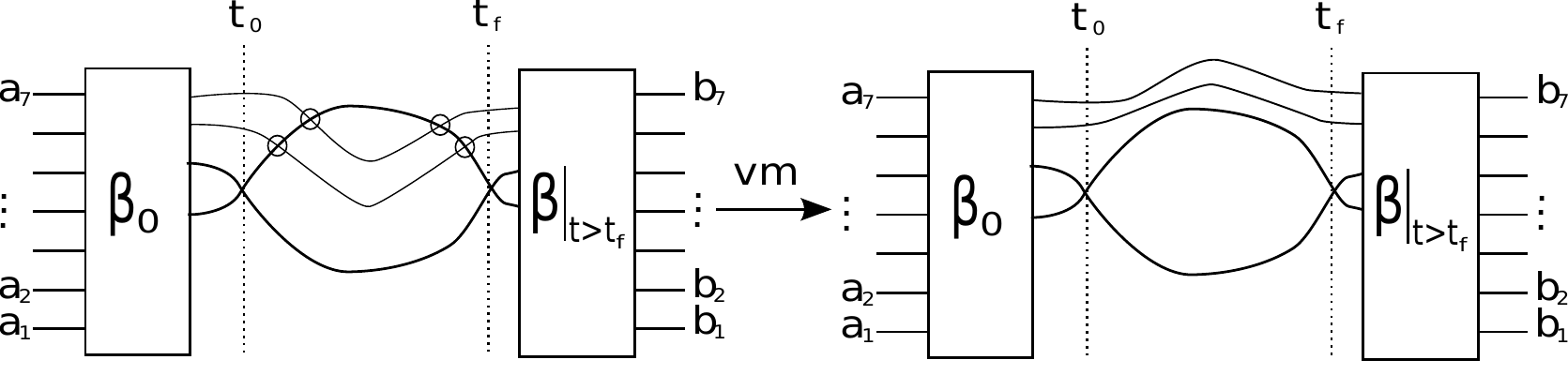}
      \caption{Case $\alpha_1(t_f)=\alpha_2(t_f)$, Lemma \ref{lemma:1}.}
      \label{lemaA5}
    \end{figure}

  Now suppose $\alpha_1(t_f)=\alpha_2(t_f)$. Then $\alpha_1$ and $\alpha_2$ form a bigon $D$. We apply the same reasoning as above in order to have only crossings on $\alpha_2$ (Figure \ref{lemaA5}).  Suppose that $m_2\neq 0$, then it necessarily is even (as each strand entering the bigon must go out by $\alpha_2$). We can apply $\frac{m}{2}$ moves of type $V2$ to get rid of the crossings in $\alpha_2$. But this is a contradiction as in each inductive step we are reducing the diagram by all the possible $V2$ moves in it. Therefore we can chose $\beta'$ such that there are no crossings in $(t_0,t_f) \times I$.  With this we complete the proof of the lemma.
\end{proof}

\begin{corollary}\label{cor:2}
      Let $\beta$ be a virtual braid diagram and let $\alpha_1, \alpha_2$ be two primitive arcs of $\beta$ such that: 
	\begin{enumerate}
		\item The arcs $\alpha_1$ and $\alpha_2$ start in the same point, say $p$ (thus a crossing, it may be virtual or regular).
		\item The arcs $\alpha_1$ and $\alpha_2$ end at the same time, say $t_f$.
	\end{enumerate}
	Then there exists a virtual braid diagram $\beta'$ virtually equivalent to $\beta$ such that:
	\begin{enumerate}
		\item If $\beta_1$ and $\beta_2$ are the strands corresponding to $\alpha_1$ and $\alpha_2$, respectively, then up to isotopy they remain unchanged in $\beta'$ and in their restriction to $(0,\pi_1(p)=t_0)$ we add only virtual crossings. 
		\item The diagrams $\beta$ and $\beta'$ coincide for $t\geq t_f$, i.e. $\beta|_{t\geq t_f} = \beta'|_{t\geq t_f}$.
		\item Let $\alpha_1 \cap \alpha_2 = \{p=p_1,p_2,\dots, p_m \}$, numbered so that $\pi_1(p_i) < \pi_1(p_{i+1})$ for $1\leq i \leq m-1$.
		\begin{enumerate}
		  \item If $\pi_1(p_m)=t_f$, then in $(t_0, t_f)\times I$ there are no crossings except, eventually, $p_2,\dots,p_{m-1}$.	  
		  \item If $\pi_1(p_m)\neq t_f $, then in $(t_0,\pi_1(p_m))\times I$ there are no crossings except eventually $p_2,\dots,p_{m-1}$ and in $(\pi_1(p_m),t_f)\times I$ there are only virtual crossings with the corresponding upper segment of $\alpha_1$ or $\alpha_2$.
		\end{enumerate}
	\end{enumerate}  
\end{corollary}

  \begin{figure}[h]\centering
    \includegraphics[scale=0.6]{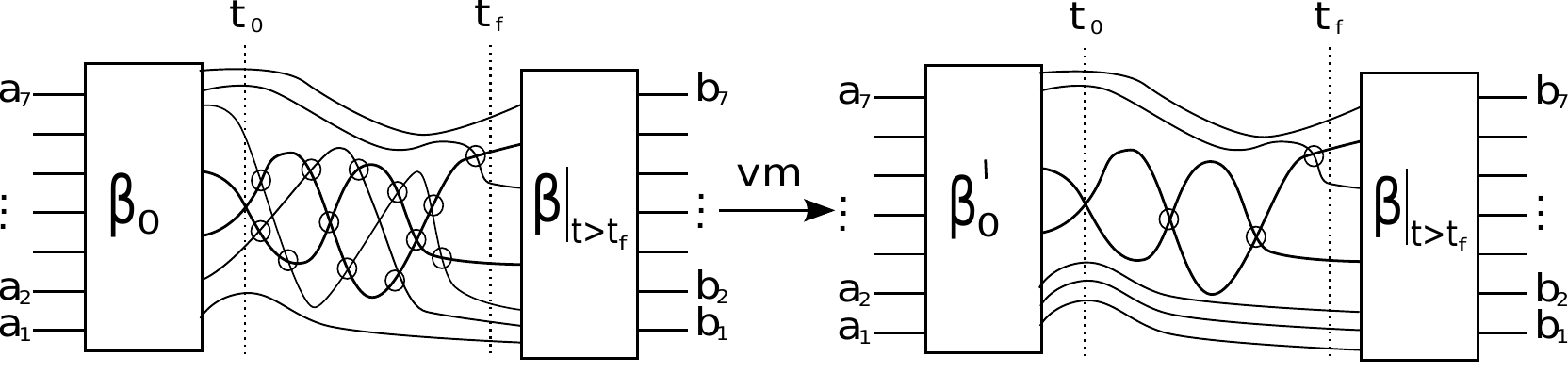} 
    \caption{Corollary \ref{cor:2}.}
  \end{figure}

\begin{proof}
  Suppose that $\pi_1(p_1)< \dots <\pi_1(p_m)$. We argue by induction on $m$. Suppose $m=1$. Then necessarily $p_1 = p$ and we have the hypothesis of Lemma \ref{lemma:1}. 

  Suppose $m>1$ and that $\alpha_1$ and $\alpha_2$ end in the same point $p_m$. Consider the restrictions of $\alpha_1$ and $\alpha_2$ to $[\pi_1(p_{m-1}), t_f]$ and apply Lemma \ref{lemma:1}. We obtain a virtually equivalent diagram $\beta'$ which does not have crossings neither on the restriction of $\alpha_1$ nor on the restriction of $\alpha_2$. Furthermore, up to isotopy the strands corresponding to $\alpha_1$ and $\alpha_2$  remain unchanged and their restrictions to $[t_0,\pi_1(p_{m-1})] $ go only through virtual crossings, i.e. they are primitive arcs whose intersection has $m-1$ points. Applying induction hypothesis on them, we have proved this case.
  
  Suppose $m>1$ and that $\alpha_1$ and $\alpha_2$ do not end in the same point. Consider the restrictions of $\alpha_1$ and $\alpha_2$ to $[\pi_1(p_{m}), t_f]$ and apply Lemma \ref{lemma:1}. We obtain a virtually equivalent diagram $\beta'$ which may have only virtual crossings with the corresponding upper segment of $\alpha_1$ or $\alpha_2$. Furthermore, up to isotopy, the strands corresponding to $\alpha_1$ and $\alpha_2$  remain unchanged and their restrictions to $[t_0,\pi_1(p_{m})]$ go only through virtual crossings, i.e. they are primitive arcs whose intersection has $m$ points and satisfies the condition of the preceding case. With this we conclude the proof.  
\end{proof}

\begin{corollary}\label{cor:3}
	Given a virtual braid diagram $\beta$ in general position. Let $c_1$ and $c_2$ be two regular crossings not related in $P(\beta)$ (Definition \ref{rmk:GaussOrder}) and such that:
	\begin{enumerate}
		\item  In the total order on $R(\beta)$ (Definition \ref{rmk:ordCross}), $D(\beta)$, $c_1> c_2$. 
		\item There is no regular crossing between $c_1$ and $c_2$ in $D(\beta)$.
	\end{enumerate}
	Then there exists a virtual braid diagram $\beta'$ virtually equivalent to $\beta$ with $c_2 > c_1$ in $D(\beta ')$, and such that there is no regular crossing between them. 
	
	 Furthermore, the diagrams $\beta$ and $\beta'$ coincide for $t>t_f$. In particular the total order on the set of elements smaller than $c_2$ in $D(\beta)$ is preserved in $D(\beta')$, i.e. $c_2>d_1>d_2$ in $D(\beta)$, then $c_1>d_1>d_2$ in $D(\beta')$.  
\end{corollary}

\begin{proof}
	Let $t_f > \pi_1(c_2)$ such that there is no crossing in $(\pi_1(c_2),t_f)\times I$ and, let $\alpha_1$ and $\alpha_2$ be the primitive arcs coming from  the regular crossing $c_1$ and finishing in $t_f$. Applying the last corollary to $\alpha_1$ and $\alpha_2$, we obtain a virtually equivalent diagram $\beta'$ such that in $(\pi_1(c_1),t_f]\times I$ there are only virtual crossings and $\beta$ remains unchanged for $t\geq t_f$. Thus $c_2 > c_1$ in $D(\beta')$ and if $c_2>d_1>d_2$ in $D(\beta)$, then $c_1>d_1>d_2$ in $D(\beta')$. 
\end{proof}

Given two orders, $R$ and $R'$ over a set $X$, we say that  $R'$ is compatible with $R$ if $R\subset R'$.

\begin{lemma}\label{lemma:2}
	Let $\beta$ be a virtual braid diagram and let $\tilde{R}$ be a total order on $R(\beta)$ compatible with $P(\beta)$. Then there exists a virtual braid diagram $\beta'$ virtually equivalent to $\beta$ such that $D(\beta')=\tilde{R}$.
\end{lemma}

  \begin{figure}[h]\centering
    \includegraphics[scale=0.6]{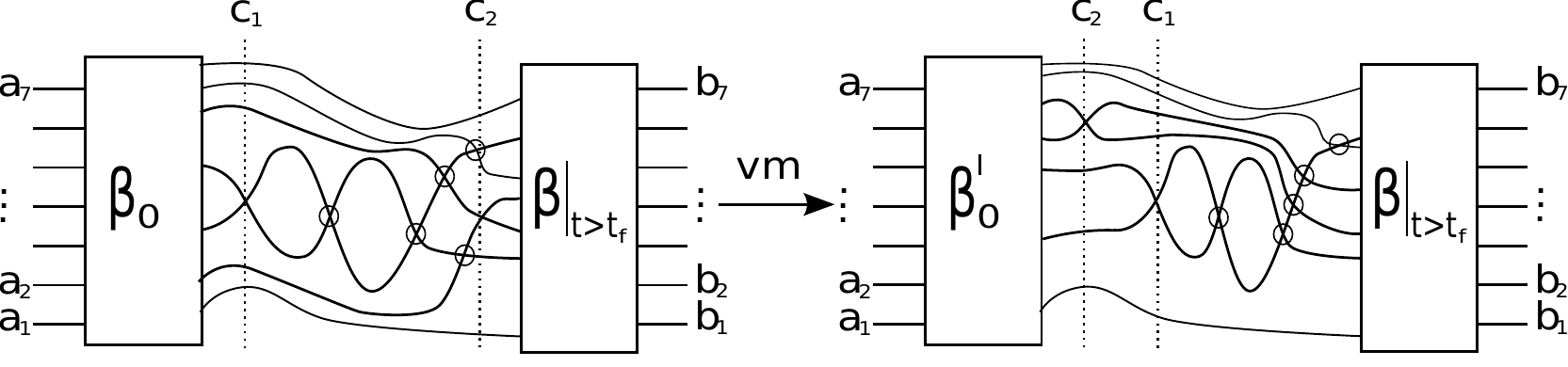} 
    \caption{Lemma \ref{lemma:2}.}
  \end{figure}

\begin{proof}
	Suppose  
	$$\tilde{R}= \{c_1 > \dots > c_m\},$$ 
	$$D(\beta) = \{d_1 > \dots > d_m\},$$
	and that $c_l = d_l$ for $l>k$, $c_k\neq d_k$  and $c_k=d_j >d_k $ in $D(\beta)$. Note that for $k\geq l > j$, $d_j$ is not related with $d_l$ in $P(\beta)$. Applying the last corollary $k-j$ times we construct a virtual braid diagram $\beta'$ virtually equivalent to $\beta$ such that if $D(\beta')=\{d_1'> \dots > d_m'\}$ then $c_l=d_l'$ for $l\geq k$. Applying this procedure inductively we obtain the lemma. 
\end{proof}

\begin{lemma}\label{lemma:3}
	Let $\beta_1$ and $\beta_2$ be two virtual braid diagrams on $n$ strands, and let $\alpha_1$ and $\alpha_2$ be two primitive arcs of $\beta_1$ and $\beta_2$ respectively, such that:
	\begin{enumerate}
		\item The extremes of $\alpha_1$ and $\alpha_2$ coincide.
		\item $\alpha_1$ and $\alpha_2$ form a bigon $D$.
		\item $\beta_1 \setminus \alpha_1$ and $\beta_2 \setminus \alpha_2$ coincide.
		\item There are no crossings in the interior of $D$.
	\end{enumerate}
	Then $\beta_1$ and $\beta_2$ are virtually equivalent by isotopy and moves of type $V2$.
\end{lemma}

\begin{proof}
	First note that each strand entering $D$ must go out. Take a strand $\alpha$ entering $D$ and suppose it is innermost. If it goes out by the same side, as there are no crossings in the interior of $D$, then we can apply a move of type $V2$ and eliminate the two virtual crossings. Therefore we can suppose that each strands entering by one side goes out by the other. Apply an isotopy following the strands crossing the bigon (if there are any) in order to identify the two primitive arcs.
\end{proof}

\begin{lemma}\label{lemma:4}
	Let $\beta_1$ and $\beta_2$ be two virtual braid diagrams on $n$ strands, and let $\alpha_1$ and $\alpha_2$ be two primitive arcs of $\beta_1$ and $\beta_2$ respectively, such that:
	\begin{enumerate}
		\item The extremes of $\alpha_1$ and $\alpha_2$ coincide.
		\item $\alpha_1$ and $\alpha_2$ form a bigon $D$.
		\item $\beta_1 \setminus \alpha_1$ and $\beta_2 \setminus \alpha_2$ coincide.
	\end{enumerate}
	Then $\beta_1$ and $\beta_2$ are virtually equivalent. 
\end{lemma}

 \begin{figure}[h]\centering
    \includegraphics[scale=1.0]{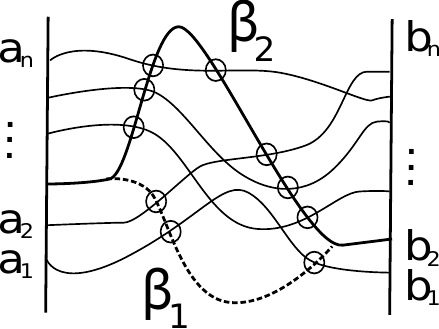} 
    \caption{Lemma \ref{lemma:4}.}
  \end{figure}

\begin{proof}
	Call $p$ and $q$ the starting and ending points of $\alpha_1$, and set $\beta= \beta_1 \setminus \alpha_1$. Let $m$ be the number of crossings inside $D$. We argue by induction on $m$. If $m=0$ we apply the last lemma. 
	
	Suppose $m\geq 1$ and let $c$ be the leftmost crossing inside $D$. Choose $t_0$ and $t_1$ so that $\pi_1(p)<t_0<t_1 < \pi_1(c)$ and so that there are no crossings in $D\cap ((t_0,\pi_1(c)) \times I)$. Draw a line $a'$ joining $\alpha_1(t_0)$ and $\alpha_2(t_1)$. Note that $a'$ intersects the two incoming strands that compose $c$. To the crossings of $\beta$ with $a'$ assign virtual crossings. Consider the following primitive arcs:
	
\begin{align*}
	c_1 &=\alpha_1|_{[\pi_1(p),t_0] }*a'  	& 	c_3 &=a'*\alpha_2|_{[t_1,\pi_1(q)]  }		\\
	c_2 &=\alpha_2|_{[\pi_1(p),t_1] }		&	c_4 &=\alpha_1|_{[t_0,\pi_1(q)]  }.				
\end{align*}

 \begin{figure}[h]\centering
    \includegraphics[scale=0.8]{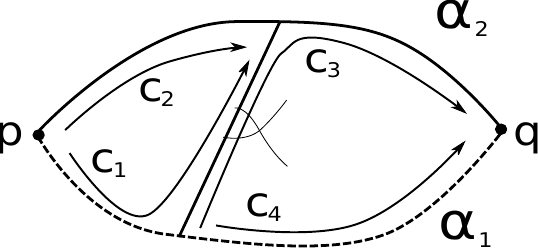} 
    \caption{Construction in proof of Lemma \ref{lemma:4}.}
  \end{figure}

Note that $c_1$ and $c_2$ form a bigon that has no crossing in its interior, so we apply the last lemma to $\beta_2'= (\beta_2\cup c_1 )\setminus c_2$ and $\beta_2$. 

On the other hand, take the bigon $D'$ formed by $c_3$ and $c_4$. $D'$ has the same crossings as $D$, and $c$ is still the leftmost crossing in $D'$. By construction of $\beta_2'$ we can apply a move of type $V3$, $M$ or $M'$ to move $a'$ to the other side of $c$. Call the obtained virtual braid diagram $\beta_1'$. Then the bigon formed between $\beta_1'$ and $\beta_1$ has $m-1$ crossings in its interior.  Applying the induction hypothesis to $\beta_1$ and $\beta_1'$ we conclude that 
$$ \beta_1 \sim_{vm} \beta_1' \sim_{vm} \beta_2' \sim_{vm} \beta_2,$$
which proves the lemma.
\end{proof}

\begin{corollary}\label{cor:4}
	Let $\beta_1$ and $\beta_2$ be two virtual braid diagrams on $n$ strands, and let $\alpha_1$ and $\alpha_2$ be two primitive arcs of $\beta_1$ and $\beta_2$, respectively, such that:
	\begin{enumerate}
		\item The extremes of $\alpha_1$ and $\alpha_2$ coincide.
		\item $\beta_1 \setminus \alpha_1$ and $\beta_2 \setminus \alpha_2$ coincide.
	\end{enumerate}
	Then $\beta_1$ and $\beta_2$ are virtually equivalent.
\end{corollary}

\begin{proof}
	Without loss of generality we can suppose that $\alpha_1$ intersects transversally $\alpha_2$ in a finite number of points. In this case they form a finite number of bigons. Apply the previous lemma to each one. 
\end{proof}

Now we are able to complete the proof of Theorem \ref{thm:1}. We have already proved (1) in Lemma \ref{lemma:0}. In Lemma \ref{lemma:01} we have shown that if $\beta \sim_{vm} \beta'$ then $G(\beta)=G(\beta')$.  It remains to prove that if $\beta, \beta' \in VBD_n$ are so that $G(\beta) = G(\beta')$, then $\beta \sim_{vm} \beta'$. Set $g=G(\beta)=G(\beta')$.


Let $\tilde{R}$ be a total order of $R(G)$, compatible with $P(G)$. By Lemma \ref{lemma:2} there exist two virtual braid diagrams, $\alpha$ and $\alpha'$, virtually equivalent to $\beta$ and $\beta'$ respectively and such that $D(\alpha)=\tilde{R}=D(\alpha')$. As $D(\alpha)=D(\alpha')$ we can suppose that the regular crossings of $\alpha$ and $\alpha'$ coincide (if not, move them by an isotopy to make them coincide). In this case $\alpha$ and $\alpha'$ differ by joining arcs. 

Suppose $\alpha$ has $m$ joining arcs. As the regular crossings of $\alpha$ and $\alpha'$ coincide, we can suppose that the corresponding joining arcs of $\alpha$ and $\alpha'$ begin and end at the same points. Apply Corollary \ref{cor:4} $m$ times, in order to make that each of the corresponding joining arcs coincide. We conclude that $\alpha$ is virtually equivalent to $\alpha'$ and thus $\beta$ and $\beta'$. 

\subsection{Virtual braids as Gauss diagrams}

The aim of this section is to establish a bijective correspondence between virtual braids and certain equivalence classes of braid-Gauss diagrams. We also give the group structure on the set of virtual braids in terms of Gauss diagrams, and use this to prove a presentation of the pure virtual braid group. 

\begin{definition}
Let $g$ and $g'$ be two Gauss diagrams. A {\it Gauss embedding} is an embedding $\varphi: g'\rightarrow g$ which send each interval of $g'$ into a subinterval of $g$, and which sends each arrow of $g'$ to an arrow of $g$ respecting the orientation and the sign. Note that there is no condition on the permutations associated to $g'$ and $g$ in the above definition. We shall say that $g'$ is embedded in $g$ if a Gauss embedding of $g'$ into $g$ is given. 
\end{definition}

Let $g'$ be a Gauss diagram of $n$ strands, so that it is embedded in $g$ by sending the interval $i$ to a subinterval of the interval $k_i$ of $g$, we say that the embedding is of type $(k_1,\dots,k_n)$. 

Consider the three Gauss diagrams presented in Figure \ref{fig:GdEmbed}. Note that $g_1$ is embedded in $g_2$ by an embedding of type $(2,1)$, and $g_3$ is embedded in $g_2$ by an embedding of type $(1,2,3)$.
\begin{figure}[h]\centering
	    \includegraphics[scale=0.7]{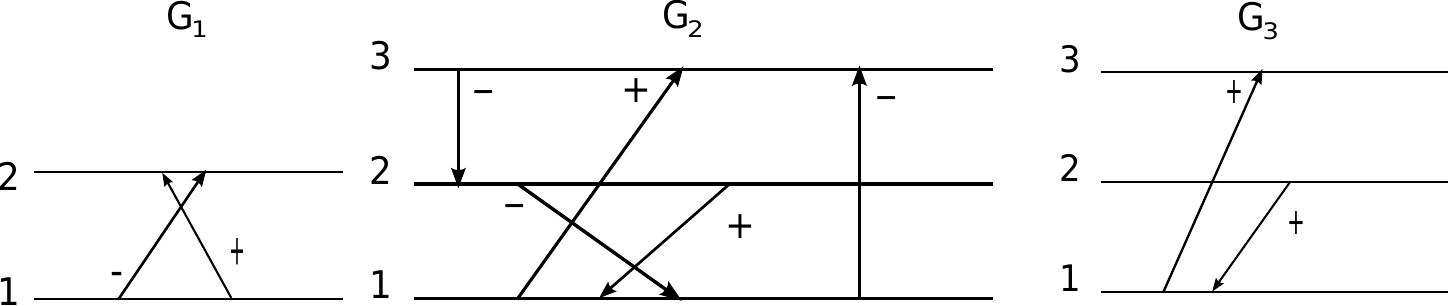} 
    \caption{$G_1$ is $(1,2)-$embedded in $G_2$, and $G_3$ is $(1,2,3)-$embedded in $G_2$.}
    \label{fig:GdEmbed}
  \end{figure}

By performing an {\it $\Omega 3$ move} on a braid Gauss diagram $g$, we mean choosing an embedding in $g$ of the braid Gauss diagram depicted on the left hand side of Figure \ref{fig:O3} (or on the right hand side of Figure \ref{fig:O3}), and replacing it by the braid Gauss diagram depicted on the right hand side of Figure \ref{fig:O3} (resp. on the left hand side of Figure \ref{fig:O3}).

\begin{figure}[h]\centering
	    \includegraphics[scale=0.8]{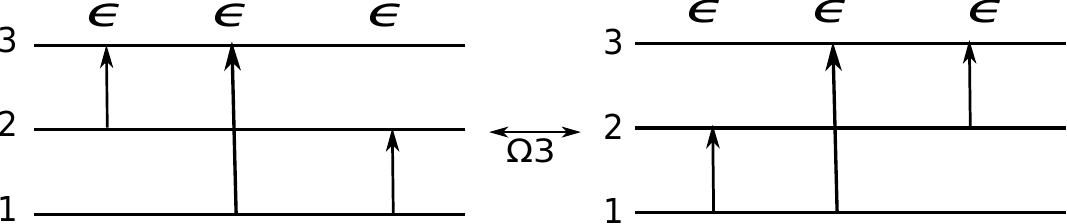} 
    \caption{$\Omega 3$ move on Gauss diagrams, with $\epsilon \in \{\pm 1\}$.}
    \label{fig:O3}
  \end{figure}

Let $g$ be a Gauss diagram with $n$ strands and $i,j,k\in \{1,\dots,n\}$ with $i<j<k$. The six different types of embeddings of the Gauss diagram in Figure \ref{fig:O3} in $g$ are illustrated in Figures \ref{fig:O3a}, \ref{fig:O3b} and \ref{fig:O3c}. According to the type of embedding the $\Omega 3$ move is called $\Omega 3$ move of type $(k_1,k_2,k_3)$.

\begin{figure}[h]\centering
	    \includegraphics[scale=0.55]{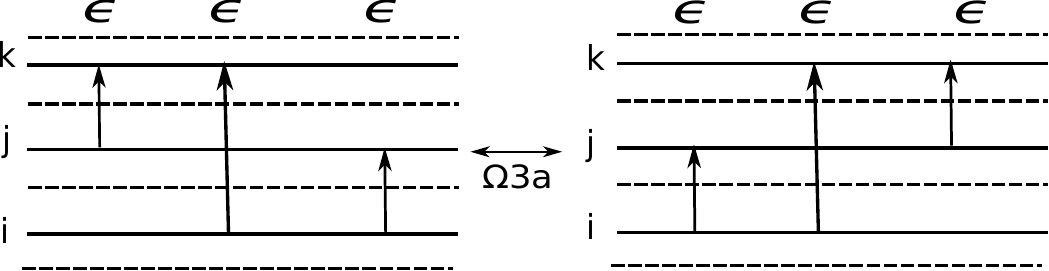} \quad
  	    \includegraphics[scale=0.55]{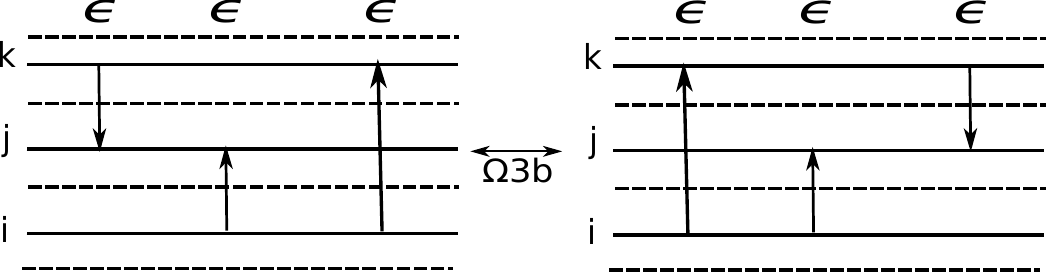} 
    \caption{$\Omega 3$ moves of type $(i,j,k)$ and $(i,k,j)$.}
    \label{fig:O3a}
  \end{figure}

\vspace{-8pt}
\begin{figure}\centering
	    \includegraphics[scale=0.55]{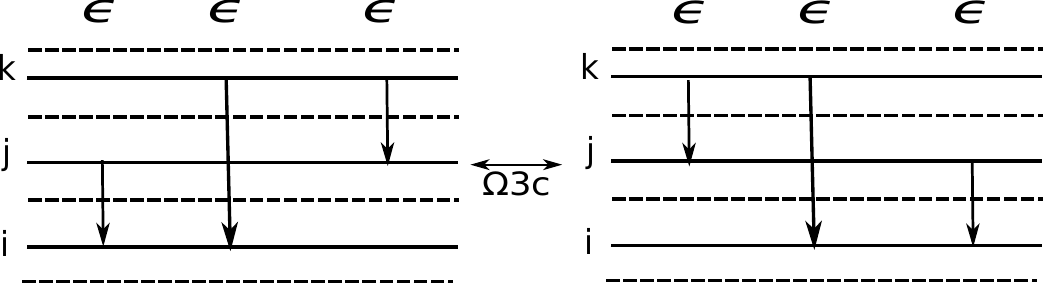} \quad
  	    \includegraphics[scale=0.55]{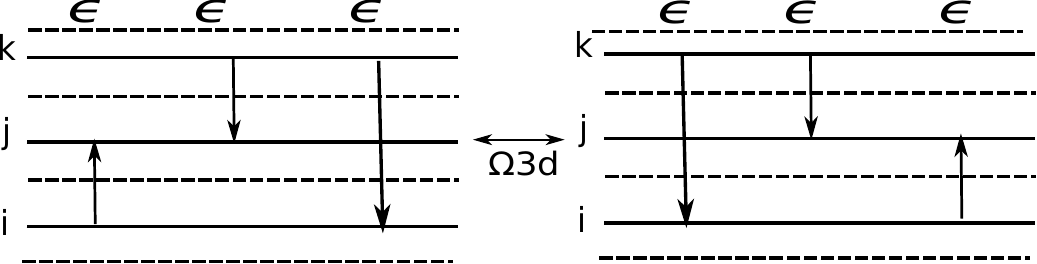} 
    \caption{$\Omega 3$ moves of type $(j,i,k)$ and $(j,k,i)$.}
    \label{fig:O3b}
  \end{figure}

\vspace{-8pt}
\begin{figure}\centering
	    \includegraphics[scale=0.55]{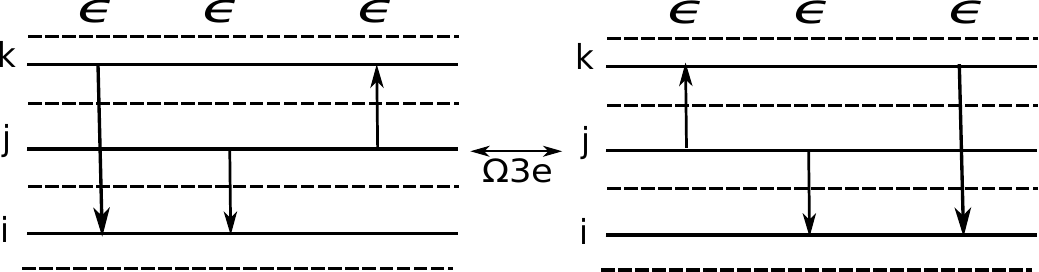} \quad
  	    \includegraphics[scale=0.55]{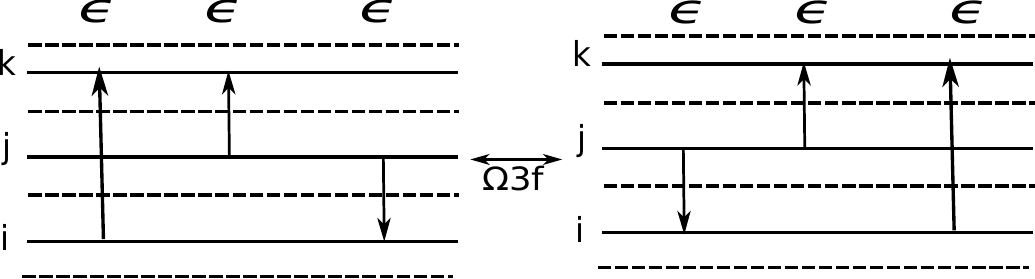} 
    \caption{$\Omega 3$ moves of type $(k,i,j)$ and $(k,j,i)$.}
\label{fig:O3c}
  \end{figure}

Similarly, by performing an {\it $\Omega 2$ move} on a braid Gauss diagram $g$, we mean choosing an embedding in $g$ of the braid Gauss diagram depicted on the left hand side of Figure \ref{fig:O2} (or on the right hand side of Figure \ref{fig:O2}), and replacing it by the braid Gauss diagram depicted on the right hand side of Figure \ref{fig:O2} (resp. on the left hand side of Figure \ref{fig:O2}). 

\begin{figure}[h]\centering
	    \includegraphics[scale=0.8]{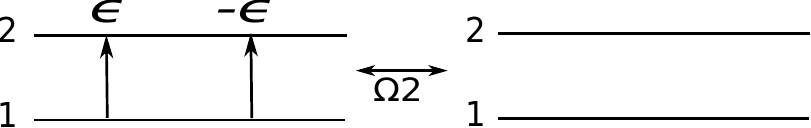} 
    \caption{$\Omega 2$ move on Gauss diagrams.}
    \label{fig:O2}
  \end{figure}

In this case there are only two types of embeddings. They are illustrated in Figure \ref{fig:O2a}.
\begin{figure}[h]\centering
	    \includegraphics[scale=0.7]{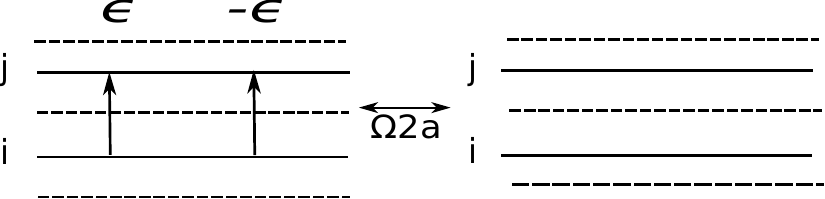} \quad
	    \includegraphics[scale=0.7]{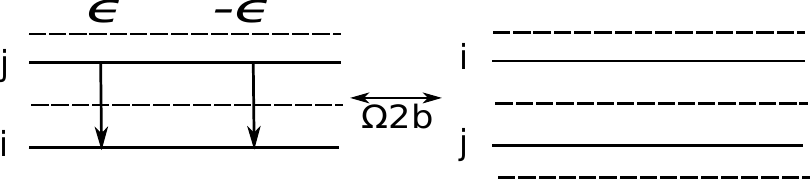} 
    \caption{$\Omega 2$ moves of type  $(i, j)$ and $(j, i)$.}
    \label{fig:O2a}
  \end{figure}

\begin{definition}
 The equivalence relation generated by the $\Omega 2$ and the $\Omega 3$ moves in the set of braid Gauss diagrams is called \textit{Reidemeister equivalence}. The set of equivalence classes of braid Gauss diagrams is denoted by $bG_n$.
\end{definition}

\begin{proposition}\label{prop:bGVb}
  There is a bijective correspondence between $bG_n$ and $VB_n$. 
\end{proposition}

\begin{proof}
  By Theorem \ref{thm:1} we know that there is a bijective correspondence between the set of virtually equivalent virtual braid diagrams and the braid Gauss diagrams. Therefore we need to prove that if two virtual braid diagrams are related by a Reidemeister move then their braid Gauss diagrams are Reidemeister equivalent, and that if two braid Gauss diagrams are related by an $\Omega 2$ or an $\Omega 3$ move then their virtual braid diagrams are virtually Reidemeister equivalent.
  
  Let $\beta$ and $\beta'$ be two virtual braid diagrams that differ by a Reidemeister move. Suppose that they are related by a $R3$ move, and that the strands involved in the move are $a$, $b$, and $c$, with $a,b,c \in \{1,\dots,n\}$. Then, up to isotopy we can deform the diagrams so that they coincide outside the subinterval  $I_0:=[t_0,t_f]\subset I$,  and in $I_0$ there are only the crossings involved in the $R3$ move. In $I_0$ the diagrams look as in Figure \ref{fig:R3label}. Thus, their braid-Gauss diagrams coincide outside $I_0$ and in $I_0$ they differ by an $\Omega 3$ move of type $(a,b,c)$. The case $R2$ is proved in the same way. 

\begin{figure}[h]\centering
	    \includegraphics[scale=0.6]{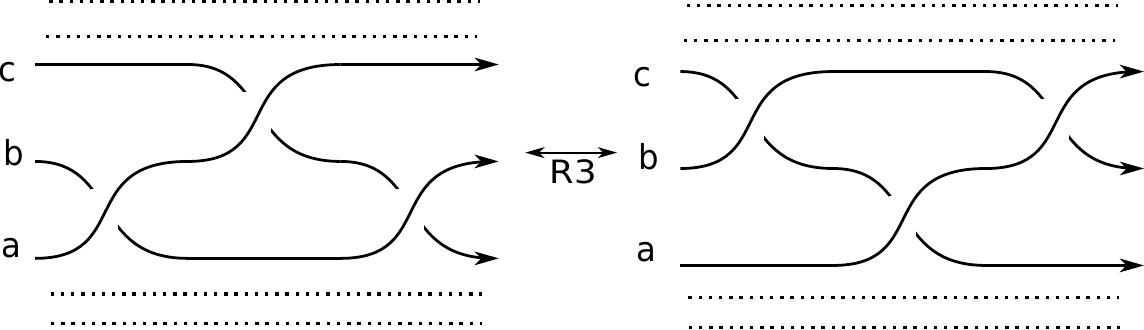} \\
    \caption{A labelled $R3$ move.}
    \label{fig:R3label}
  \end{figure}
 
   Now, let $g$ and $g'$ be two braid Gauss diagrams and let $a,b,c\in \{1,\dots,n\}$ be pairwise different. Suppose that $g$ and $g'$ are related by an $\Omega 3$ move of type $(a,b,c)$. There exists a subinterval $I_0=[t_0,t_f]\subset I$, that contains only the three arrows involved in the $\Omega 3$ move. There exists a virtual braid diagram $\beta$, representing $g$, that in the subinterval $I_0$ it looks as the left hand side (or the right hand side) of Figure \ref{fig:R3label}. By performing an $R3$ move on $\beta|_{I_0}$, we obtain a virtual braid diagram $\beta'$. Their braid Gauss diagrams coincides outside $I_0$ and in $I_0$ the differ by an $\Omega 3$ move of type $(a,b,c)$, i.e. $G(\beta')= g'$.  
   

\end{proof}

\subsection{Presentation of $PV_n$.}

 Recall that $VB_n$ has a group structure, with the product given by the concatenation of the diagrams (Remark \ref{rmk:VBn}). By Proposition \ref{prop:bGVb}, $bG_n$ has a group structure induced by the one on $VB_n$. 
 
 A presentation of the {\it pure virtual braids} was given by Bardakov \cite{Ba2004}. We present an alternative proof by means of the braid Gauss diagrams. 

 Recall that the symmetric group, $S_n$, has the next presentation:
 \begin{itemize}
 	\item Generators: $t_1,\dots,t_{n-1}$.
 	\item Relations:\[\begin{split}
				t_{i}t_{i+1}t_{i}=t_{i+1}t_{i}t_{i+1} 	&\quad \text{ for } 1\leq i \leq n-2 \\
				t_{i}t_{j}=t_{j}t_{i}		 			&\quad \text{ for } |i-j|\geq 2. \\
				t_{i}^2=1		 						&\quad \text{ for } 1\leq i \leq 2.
			 \end{split}\]
 \end{itemize}

 From the presentation of $VB_n$ (Remark \ref{rmk:VBn}), there exists an epimorphism $\theta_P : VB_n \rightarrow S_n$, given by 
$$\theta_P(\tau_i) = t_i = \theta_P(\sigma_i) \quad \text{for } 1\leq i \leq n-1. $$

 The kernel of $\theta_P$ is called the \textit{pure virtual braid group} and is denoted by $PV_n$. The elements of this group correspond to the virtual braids diagrams whose strands begin and end in the same marked point, i.e. the permutation associated to its braid Gauss diagram is the identity. 

 On the other hand, a braid Gauss diagram is composed by the next elements:
\begin{enumerate}
  \item A finite ordered set of $n$ intervals, say $I_1\sqcup I_2 \sqcup \dots \sqcup I_n$.
  \item A finite set of arrows connecting the different intervals, so that to each arrow corresponds a different time.
  \item A function assigning a sign, $\{ \pm 1 \}$, to each arrow. 
  \item A permutation, $\sigma \in S_n$, labelling the endpoint of each interval.
\end{enumerate}

  Denote by $X_{i,j}^{\epsilon}$ the arrow from the interval $i$ to the interval $j$ with sign $\epsilon\in \{\pm 1\}$. Let $$X= \{X_{i,j}^{\epsilon} \; | \; 1\leq i\neq j \leq n \; , \; \epsilon\in \{\pm 1\} \}$$ and denote by $X^*$ the set of all words in $X$ union the empty word, denoted by $e$.

 Given a braid Gauss diagram, $g$, its arrows have a natural order induced by the parametrization of the intervals. Let $W\in X^*$ be the word given by the concatenation of the arrows in $g$, according to the order in which they appear, and $\sigma\in S_n$ its associated permutation.  Thus any braid Gauss diagram can be expressed as $g= (W,\sigma)$. We denote $\bar{e}:= (e,Id_{S_n})$.
 
\begin{proposition}(Bardakov \cite{Ba2004})\label{prop:PVn}
	The group $PV_n$ has the following presentation:
	\begin{itemize}
		\item Generators: $A_{i,j}$ with $1\leq i \neq j \leq n$.
		\item Relations: \[\begin{split}
				A_{i,j}A_{i,k}A_{j,k} = A_{j,k}A_{i,k}A_{i,j} &\quad \text{for }i,j,k \text{  distinct}. \\
				A_{i,j}A_{k,l} = A_{k,l}A_{i,j} &\quad \text{for }i,j,k,l \text{  distinct}. \\
			 \end{split}\]
	\end{itemize}
\end{proposition}

\begin{proof}
	Given a pure virtual braid diagram $\beta$, its braid Gauss diagram is given by $G(\beta) = (W ,Id_{S_n})$. Thus any pure virtual braid diagram may be expressed as an element in $X^*$. 
	
	Recall that, as elements of $bG_n$, the braid Gauss diagrams are related by three different moves (and its inverses) on the subwords of any word in $X^*$: 
	\begin{enumerate}
		\item Reparametrization: 
		   $$X_{i,j}^{\epsilon_
1} X_{k,l}^{\epsilon_2}= X_{k,l}^{\epsilon_2}X_{i,j}^{\epsilon_1} \text{ for } i,j,k,l \text{ distinct and } \epsilon_1,\epsilon_2 \in \{\pm 1\}.$$ 
		\item The $\Omega 2$ move: 
		$$X_{i,j}^{\epsilon}X_{i,j}^{- \epsilon}= e \text{ for } i,j \text{ distinct and }\epsilon\in\{\pm 1\}.$$
		\item The $\Omega 3$ move:
		$$X_{i,j}^{\epsilon}X_{i,k}^{\epsilon}X_{j,k}^{\epsilon} = X_{j,k}^{\epsilon}X_{i,k}^{\epsilon}X_{i,j}^{\epsilon} \text{ for } i,j,k \text{ distinct and  } \epsilon\in\{\pm 1\}.$$	
	\end{enumerate}

	Denote by $PG_n$ the set of equivalence classes of $X^*$. Note that $PG_n$ has the structure of group with the product defined as the concatenation of the words. On the other hand $G: PV_n \rightarrow PG_n$ is an homomorphism,  i.e. $G(\beta_1 \beta_2) = G(\beta_1)G(\beta_2)$.  By Proposition \ref{prop:bGVb}, $G$ is a bijection. Consequently it is an isomorphism. 

	Let $\Gamma$ be the group with presentation as stated in the proposition. Let $\Psi : \Gamma \rightarrow PG_n$ be given by $$\Psi (A_{i,j}) = X_{i,j},$$ and let $\Phi : PG_n \rightarrow \Gamma$ be given by $$\Phi (X_{i,j}^{\epsilon}) = A_{i,j}^{\epsilon}.$$ 
	
	Note that $\Phi$ and $\Psi$ are well-defined homomorphisms and furthermore $\Psi\circ \Phi = Id_{PG_n}$ and $\Phi\circ \Psi = Id_{\Gamma}$.  Consequently $PV_n$ has the presentation stated in the proposition. 
	\end{proof}

%

\section{Abstract braids}

The aim of this section is to establish a topological representation of virtual braids.

\begin{definition}
	A {\it abstract braid diagram on $n$ strands} is a quadruple $\bar{\beta}=(S,f,\beta,\epsilon)$, such that:
	\begin{enumerate}
		\item $S$ is a connected, compact and oriented surface.		
		\item The boundary of $S$ has only two connected components, i.e. $\partial S = C_0 \sqcup C_1$, with $C_0 \approx S^{1} \approx C_1$. They are called \textit{distinguished boundary components}. 
		\item Each boundary component of $S$ has $n$ marked points, say $K_0 = \{a_1,\dots,a_n \} \subset C_0$ and $K_1 = \{b_1,\dots,b_n \} \subset C_1$. Such that: 
			\begin{enumerate}
				\item The elements of $K_0$ and $K_1$ are linearly ordered.
				\item Let $\kappa_0: S^{1} \rightarrow C_0$ and $\kappa_1 : S^{1} \rightarrow C_1$ be parametrizations of $C_0$ and $C_1$ compatible with the orientation of $S$. Up to isotopy we can put $a_k= \kappa_0(e^{\frac{2\pi i}{k}})$ and $b_k = \kappa_1(e^{-\frac{2\pi i}{k}})$ for $k\in \{1, \dots, n\}$.  
			\end{enumerate} 
		\item $f$ is a smooth function, $f: S\rightarrow I$ such that $f^{-1}(\{ 0 \}) = C_0$ and $f^{-1}(\{ 1 \}) = C_1$.
		\item $\beta$ is an $n$-tuple of curves $\beta = (\beta_1,\dots,\beta_n)$ with
		   \begin{enumerate}
		     \item For $k=1,\dots, n$, $\beta_k : I \rightarrow S$.
		     \item For $k=1,\dots, n$, $\beta_k(0)=a_k$.
		     \item There exists $\sigma \in S_n$ such that $$\beta_k(1)= b_{\sigma(k)},$$ for all $k\in \{1,\dots, n\}$.
		     \item For $k=1,\dots, n$ and $t\in I$, $f\circ \beta_k(t)= t$.
		     \item The $n$-tuple of curves $\beta$ is in general position, i.e. there are only transversal double points, called crossings.
		   \end{enumerate}
		\item  Similarly to Defintion \ref{def:vbd}, denote by $R(\beta)$ the set of crossings of $\beta$. Then $\epsilon$ is a function,  $$\epsilon : R(\beta) \rightarrow \{\pm 1\}.$$
	\end{enumerate}
\end{definition}

From now on we fix $n\in \N$ and we say abstract braid diagram instead of abstract braid diagram on $n$ strands.

 \begin{figure}[h]\centering
    \includegraphics[scale=0.7]{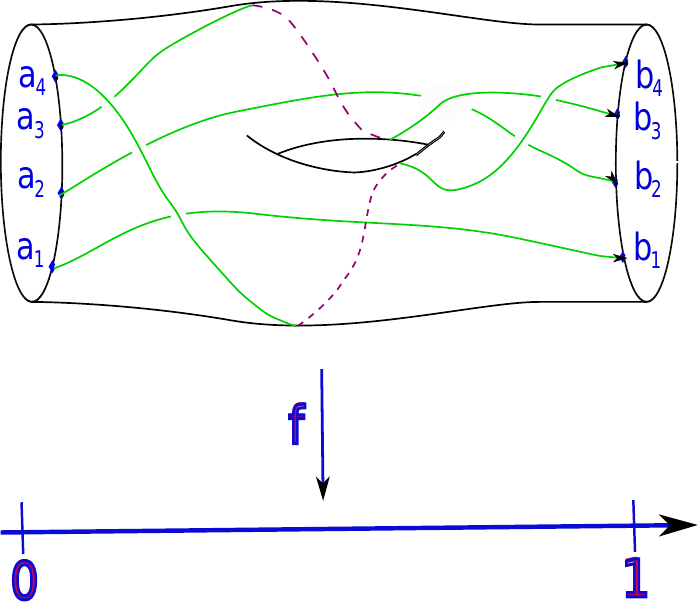} 
    \caption{An abstract braid diagram on four strands.}
 \end{figure}

\begin{definition}
  An {\it isotopy} of abstract braid diagrams is a family of abstract braid diagrams $G = \{ \bar{\beta}^{s}= (S,f^{s},\beta^{s},\epsilon^{s}) \}_{s\in I}$, such that:
  \begin{enumerate}
    \item For all $s\in I$, $K_0^s = K_0^0$ and $K_1^s = K_1^0$.
    \item For all $k\in \{1,\dots, n\}$, $H_k$ is continuous, where:
     \[ \begin{split}
         H_k : & I \times I \rightarrow S \\
               & (s,t) \mapsto \beta_k^{s}(t).
         \end{split}\]
    \item The function $H$ is smooth, where:
     \[ \begin{split}
         H : & I \times S \rightarrow I \\
               & (s,x) \mapsto f^{s}(x).
         \end{split}\]
    \item The function $\epsilon^{s} : R(\beta^{s}) \rightarrow \{\pm 1\}$ remains invariant, i.e. $\epsilon^{s}= \epsilon^0$, for all $s\in I$.
  \end{enumerate}
  We say that $\bar{\beta}^{0}$ and $\bar{\beta}^{1}$ are isotopic and we denote it by $G:\bar{\beta}^{0}\simeq \bar{\beta}^{1}$. 
\end{definition}

\begin{remark}
 The isotopy relation is an equivalence relation on the set of abstract braid diagrams. 
\end{remark}

\begin{definition}
 Let $\bar{\beta}=(S,f,\bar{\beta},\epsilon)$ and $\bar{\beta}'=(S',f',\bar{\beta}',\epsilon')$ be two abstract braid diagrams. We say that $\bar{\beta}$ and $\bar{\beta}'$ are \textit{compatible} if there exists a diffeomorphism $\varphi: (S,\partial S) \rightarrow (S',\partial S')$, such that $\varphi_* \bar{\beta} = (S',f\circ \varphi^{-1}, \varphi\circ \beta, \epsilon)$ is isotopy equivalent to $\bar{\beta}'$. We denote it by $\bar{\beta}\approx \bar{\beta}'$.
\end{definition}

\begin{remark}
The compatibility relation is an equivalence relation on the set of abstract braid diagrams, and the isotopy equivalence is included in the compatibility relation. We denote by $ABD_n$ the set of compatibility classes of abstract braid diagrams on $n$ strands.  
\end{remark}

%
%
%

\begin{definition}
 Given $\bar{\beta^{0}}=(S_0,f_0,\beta^{0},\epsilon_0)$ and $\bar{\beta^{1}}=(S_1,f_1,\beta^{1},\epsilon_1)$. We say that they are related by a \textit{stability move} if there exist:
 \begin{enumerate}
    \item Two disjoint embedded discs, $D_0$ and $D_1$, in $S_0\setminus \beta^0$.
    \item An embedding $\varphi: (S_0' = S_0 \setminus (D_0\cup D_1),\partial S_0) \rightarrow (S_1,\partial S_1)$, such that $S_1 \setminus \varphi(S_0') \approx S^1 \times I$.
    \item A smooth function $F:S_1 \rightarrow I$, such that:
     \begin{enumerate}
     	\item $f_0|_{S_0'} = F\circ \varphi$.
     	\item The quadruple $(S_1,F,\varphi\circ \beta^0,\epsilon_0)$ is an abstract braid.
     	\item $(S_1,F,\varphi\circ \beta^0,\epsilon_0) \approx \bar{\beta^1}$.
     \end{enumerate} 
 \end{enumerate} 
\end{definition}

\begin{definition}
 Given $\bar{\beta^{0}}=(S_0,f_0,\beta^{0},\epsilon_0)$ and $\bar{\beta^{1}}=(S_1,f_1,\beta^{1},\epsilon_1)$. We say that they are related by a \textit{destability move} or a \textit{destabilization}, if there exist:
 \begin{enumerate}
   \item An essential non-separating simple curve $C$ in $S_0\setminus \beta^0$.
   \item An embedding $\varphi : (S_0'=S_0\setminus C, \partial S_0) \rightarrow (S_1,\partial S_1)$, such that $S_1 \setminus \varphi(S_0\setminus C)$  is homeomorphic to the disjoint union of two closed discs.
   \item A smooth funcion $F:S_1 \rightarrow I$, such that:
     \begin{enumerate}
     	\item $f_0|_{S_0'} = F\circ \varphi$.
     	\item The quadruple $(S_1,F,\varphi\circ \beta^0,\epsilon_0)$ is an abstract braid.
     	\item $(S_1,F,\varphi\circ \beta^0,\epsilon_0) \approx \bar{\beta^1}$.
     \end{enumerate}  
 \end{enumerate} 
\end{definition}

 Given two abstract braid diagrams $\bar{\beta^0}$ and $\bar{\beta^1}$, if $\bar{\beta^1}$ is obtained from $\bar{\beta^0}$ from a stability move along two discs $D_0$ and $D_1$ in $S_0\setminus \beta^0$, the boundaries of $D_0$ and $D_1$ are homotopy equivalent in $S_1$. If we perform a destabilization along its homotopy class we recover $\bar{\beta^0}$, up to compatibility. 
 
 Reciprocally if $\bar{\beta^1}$ is obtained from $\bar{\beta^0}$ by a destabilization along an essential curve $C$, then we can recover $\bar{\beta^0}$, up to compatibility, with a stabilization along the two capped discs in $S_0\setminus C$, see Figure \ref{fig:stabComp}. 

 \begin{figure}[h]\centering
    \includegraphics[scale=0.52]{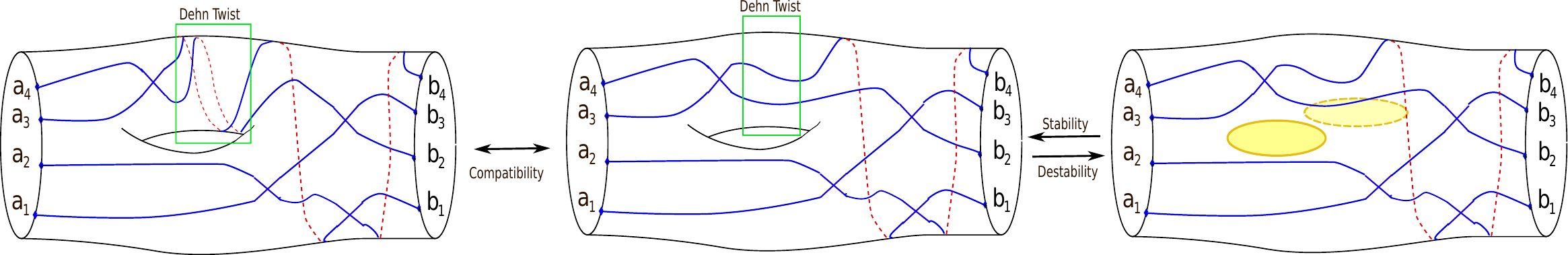} 
    \caption{Compatibility and stability equivalence.}
    \label{fig:stabComp}
 \end{figure}

\vspace{-8pt}
\begin{definition}
  The equivalence relation on the set of abstract braids generated by the stability (and destability) moves is called \textit{stability equivalence}. We denote it by $\sim_s$. 
\end{definition}
 

\begin{definition}
 Let $\bar{\beta}=(S,f,\beta,\epsilon)$ be an abstract braid, and let $C$ be an embedded simple closed curve in $S\setminus \beta$. Denote by $S_C$ the connected component of $S\setminus C$ containing $\beta$.  Let $S_C'$ be a compact, connected, oriented surface and $\varphi : S_C \rightarrow S_C'$ such that:
 \begin{enumerate}
   \item The surface $S_C'$ has only two boundary components $C_0'$ and $C_1'$. 
   \item The map $\varphi$ is an embedding such that $\varphi(C_0)=C_0'$ and $\varphi(C_1)=C_1'$.
   \item Let $k_C$ be the number of connected components of $S\setminus C$.
     \begin{enumerate}
       \item If $k_C=1$, then $S_C' \setminus \varphi(S_C)$ is homeomorphic to a disjoint union of two discs. 
       \item If $k_C=2$, then $S_C' \setminus \varphi(S_C)$ is homeomorphic to a disc. 
     \end{enumerate}
 \end{enumerate}
 Let $F_C: S_C' \rightarrow I$ be a smooth function such that $F|_{\varphi(S_C)}=f|_{S_C}$ (note that up to isotopy, this extension is unique). Then $\bar{\beta}_C=(S_C',F_C,\varphi\circ \beta, \epsilon)$ is an abstract braid.  We say that we obtain $\bar{\beta}_C$ by \textit{destabilizing} $\bar{\beta}$ along $C$, and is called a \textit{generalized destabilization}.
\end{definition}

\begin{proposition}
 Let $\bar{\beta}=(S,f,\beta,\epsilon)$ be an abstract braid, and let $C$ be an embedded simple closed curve in $S\setminus \beta$. Then $\bar{\beta}_C$ is stable equivalent to $\bar{\beta}$ by a finite number of destabilizations.
\end{proposition}

\begin{proof}
   First note that if $S\setminus C$ has only one connected component the generalized destabilization along $C$ coincides with the definition of destabilization. Thus $\bar{\beta}\sim_s \bar{\beta}_C$ by one destabilization.

   So, we can assume that $S\setminus C$ has two connected components, one of which contains $\beta$ (we call it $S_C$) and the other is a compact connected surface with one boundary component, thus it is homeomorphic to $\Sigma_{g,1}$. We will prove the proposition by induction on $g$. 
   
   If $g=0$ then $\Sigma_{g,1}$ is a disc, thus $F_C \simeq f|_{S_C}$ and consequently $\bar{\beta}_C\approx \bar{\beta}$. 
   
   If $g=1$ then $\Sigma_{1,1}$ is a torus with one boundary component, which corresponds to the curve $C$. Let $C'$ be a closed simple essential non separating curve in $\Sigma_{1,1}$. We claim that $\bar{\beta}_C \approx \bar{\beta}_{C'}$. 
   
   Note that $\Sigma_{1,1}\setminus C'$ is homeomorphic to a pair of pants (Figure \ref{fig:pants}), whose exterior boundary is the curve $C$ and whose interior boundaries correspond to the boundaries generated by cutting $S$ along $C'$. 
      
   On the other hand consider the curve $C'$ embedded in $S$. The surface $S\setminus C'$ has one connected component and two (non distinguished) boundary components. Let $S'$ be the surface obtained from $S\setminus C'$ by capping the boundary components corresponding to $C'$.  There exist a disc, $D'$, embedded in $S'$ so that its boundary corresponds to the curve $C$. Thus $S_C$ is embedded in $S_{C'}$ and $S_{C'}$ is embedded in $S'$. 
      
  Suppose $\iota: S_C \hookrightarrow S_{C'}$ and $\varphi_{C'} : S_{C'} \hookrightarrow S'$ are the embeddings. Denote $\varphi_C = \varphi_{C'}\circ \iota$.  Let $F_C: S'\rightarrow I $ be an extension of $f|_{S_{C}}$ and $F_{C'}: S'\rightarrow I$ be an extension of $f|_{S_{C'}}$. Note that $F_C$ and $F_{C'}$ differ only in the interior of the disc bounded by $C$. Consequently $F_C \simeq F_{C'}$.  From this we conclude that $\bar{\beta}_C \approx \bar{\beta}_{C'}$. Thus $\bar{\beta}\sim _s \bar{\beta}_C$ by a unique destabilization.

   \begin{figure}[h]\centering
	    \includegraphics[scale=0.4]{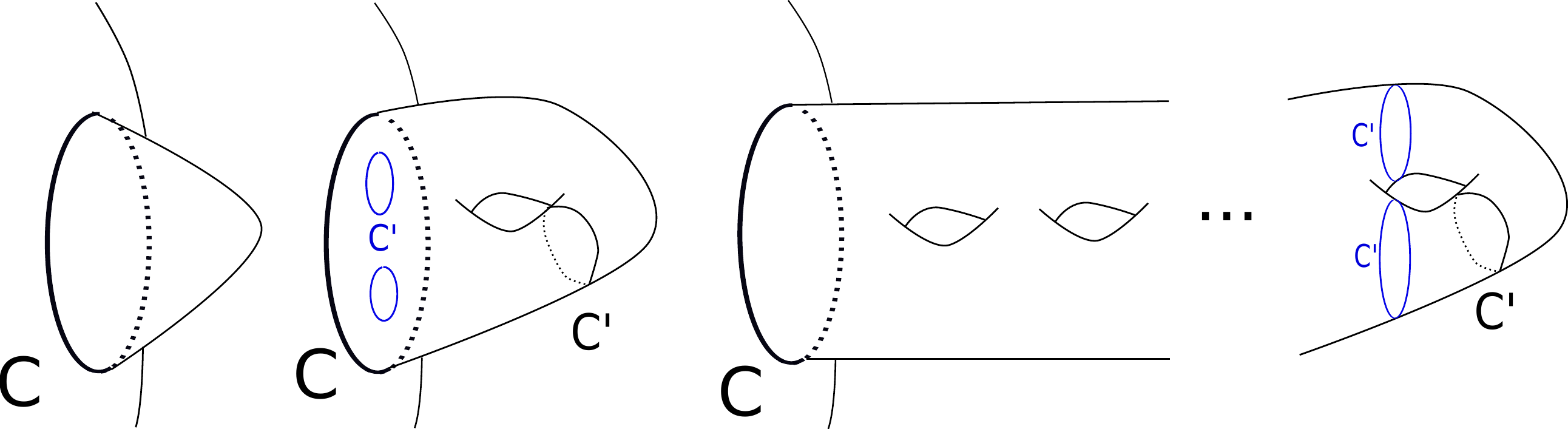} 
    \caption{Generalized destabilization along a curve $C$.}
    \label{fig:pants}
  \end{figure}

   Suppose that the proposition is true when the second connected component is homeomorphic to $\Sigma_{k,1}$. 
 
   Choose a simple essential closed curve $C$ which divides $S$ in two connected components, from which the component that does not contain $\beta$ is homeomorphic to $\Sigma_{k+1,1}$. Take a simple essential closed curve $C'$ in $\Sigma_{k+1,1}$, which is not isotopic to $C$ in $\Sigma_{k+1,1}$. Destabilize $\bar{\beta}$ along $C'$. Then, by induction, $\bar{\beta}$ is stable equivalent to $\bar{\beta}_{C'}$. The curve $C$ is still a simple closed curve in $S_{C'}\setminus \beta$, thus we can destabilize $\bar{\beta}_{C'}$ along $C$. 
   
   By induction hypothesis, the destabilization of $\bar{\beta}_{C'}$ along $C$ is stable equivalent to  $\bar{\beta}_{C'}$. Thus $\bar{\beta}$ is stable equivalent to $(\bar{\beta}_{C'})_C$. 
   
   Without loss of generality we can suppose that $(\varphi_{C'})_C = \varphi_C$, and note that $F_C$ and $(F_{C'})_C$ differ by an isotopy in the disc bounded by $C$. Consequently $(\bar{\beta}_{C'})_C \approx \bar{\beta}_C$ and $\bar{\beta}\sim_s \bar{\beta}_C$.
\end{proof}

\begin{definition}
Given two abstract braid diagrams $\bar{\beta}=(S,f,\beta,\epsilon)$ and $\bar{\beta'}=(S,f',\beta ',\epsilon ')$, we say that they are related by a \textit{Reidemeister move} or simply by an \textit{$R$-move} if, up to isotopy, $f=f'$ and there exists a neighbourhood $D$ in $S$,  homeomorphic to a disc, such that $\beta \setminus D = \beta ' \setminus D$, $\epsilon |_{\beta \setminus D} = \epsilon '|_{\beta ' \setminus D}$, and inside $D$ we can transform $\beta$ into $\beta '$ by a Reidemeister move and isotopy (Figure \ref{Rmoves}). The equivalence relation generated by the $R$-moves is called {\it Reidemeister equivalence} or simply {\it $R$-equivalence}. We denote it by $\bar{\beta } \sim_R \bar{\beta '}$.
\end{definition}

\begin{definition}
 Let $\sim$ be the equivalence relation on the abstract braid diagrams on $n$ strands generated by the compatibility, stability and Reidemeister moves.  The equivalence classes of abstract braid diagrams are called abstract braids, and the set of abstract braids is denoted by $AB_n$. 
\end{definition}

 \begin{remark}\label{rmk:bGDisoCom}
  The definition of braid Gauss diagram is extended in a natural way to the set of abstract braid diagrams. The braid Gauss diagram of an abstract braid diagram is invariant under compatility (resp. under isotopy) and stability.  

Thus, there is a well defined map from $ABD_n$ to $bGD_n$, which associates to each abstract braid diagram its braid Gauss diagram. This map is well defined up to compatibility and stability. By abuse of notation we denote the induced map still by $G$.
\end{remark}

 Recall that the set of braid Gauss diagrams is in bijective correspondence with the set of virtually equivalent virtual braid diagrams. Thus, braid Gauss diagrams are a good tool to prove that abstract braids are a good geometric interpretation of virtual braids. We present an analogous of Theorem \ref{thm:1} for abstract braid diagrams. 

 \begin{claim}\label{claim:thm2}
   The map $G: ABD_n \rightarrow bGD_n$ induces a bijection between the stable equivalence classes of abstract braid diagrams and the braid Gauss diagrams. 
  \end{claim}
  \begin{proof}
	Recall that the function is well defined from the stable and compatibility equivalence classes of Abstract braid diagrams to the braid Gauss diagrams (Remark \ref{rmk:bGDisoCom}).  
    
    Now we proof the surjectivity. Let $g \in bGD_n$. Then by Theorem \ref{thm:1} there exists a virtual braid diagram $\beta$ such that $G(\beta) = g$. For each $\beta \in VBD_n$ we can construct an abstract braid diagram $\bar{\beta}$ such that $G(\beta) = G(\bar{\beta})$ as follows.

	Let $\beta$ be a virtual braid diagram, and let $N$ be a regular neighbourhood of $\beta\cup (\{0\}\times I )\cup (\{1\}\times I)$ in $\D = I \times I$ (Figure \ref{fig:thmA0}). Note that $N$ can be seen as the union of regular neighbourhoods of each strand and of the two extremes of the virtual braid diagram. 
	
	Now consider the standard embedding of $\D$ in $\R ^3$.  Around each virtual crossing perturb the regular neighbourhoods of the strands involved in the crossing so that they do not intersect, as pictured in Figure \ref{fig:thmA0}. To the regular neighbourhood of each extreme attach a ribbon so that each extreme is now a cylinder, as in Figure \ref{fig:thmA0}.  In this way we obtain a compact oriented surface, $S'$, with more than the two distinguished boundary components. Consider the function $f:S' \rightarrow [0,1]$ defined by the projection on the first coordinate in $\R ^3$. 
	
	As $S'$ is compact, connected and oriented, it is diffeomorphic to $\Sigma_{g,b}$. We can cap all the non-distinguished boundary components in order to obtain a surface $S$ that has only the distinguished boundary components. There exists an embedding $\varphi: S'\rightarrow S$ and a smooth function $F:S\rightarrow I$, such that $f = F\circ \varphi$. In this way we have constructed an abstract braid diagram $\bar{\beta}=(S,F,\beta, \epsilon)$ such that $G(\beta)=G(\bar{\beta})$. From this we conclude that the function $G$ is surjective. 

 \begin{figure}[h]\centering
    \includegraphics[scale=0.64]{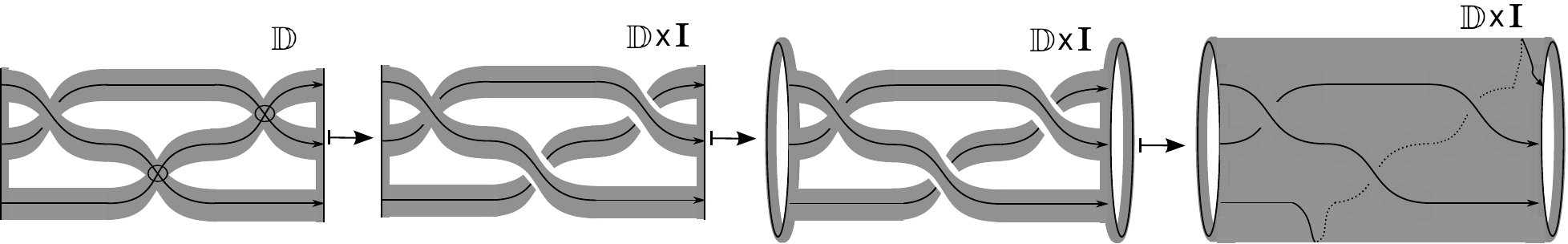} 
    \caption{Construction of $\bar{\beta}$ from $\beta$ such that $G(\bar{\beta}) = G(\beta)$.}
    \label{fig:thmA0}
  \end{figure}    
  
  Now to prove injectivity of the induced function, let $\bar{\beta} = (S,f,\beta,\epsilon)$ and $\bar{\beta}' = (S',f',\beta',\epsilon')$ be two abstract braid diagrams such that $G(\bar{\beta}) = G(\bar{\beta}')$. We claim that $\bar{\beta}$ is stable equivalent to $\bar{\beta}'$.  
  
  Note that $G(\bar{\beta}) = G(\bar{\beta'})$ implies that the graph given by $\Gamma = C_0\cup \beta \cup C_1 \subset S$ is homeomorphic to $\Gamma' =C_0'\cup \beta ' \cup C_1' \subset S'$. Consider a regular neighbourhood of $\Gamma$ in $S$, $N$, and a regular neighbourhood of $\Gamma'$ in $S'$, $N'$. Thus there exists an homeomorphism $\varphi: N\rightarrow N'$, with $\varphi(\Gamma)= \Gamma'$.

  As $N$ is homeomorphic to $\Sigma_{g,k+2}$, it has $k$ non-distinguished boundary components. We can cap the $k$ non-distinguished boundary components of $N$ to obtain a surface $\Sigma$ that has only the two distinguished boundary components. There exists an embedding $\iota: N \rightarrow \Sigma$ and a smooth function $F:\Sigma \rightarrow I$ such that $f|_{N} = F\circ \iota$. In this way we have constructed an abstract braid diagram $\bar{\alpha} = (\Sigma, F,\iota\circ \beta,\epsilon)$ stable equivalent to $\bar{\beta}$.

  On the other hand, note that $f|_{N}$ is homotopic to $g = f'\circ \varphi$ and as $\Sigma \setminus N$ is a disjoint union of circles, then we can extend $g$ to $\Sigma$ so that it is homotopy equivalent to $F$. Thus without loss of generality we can suppose that $f|_N= f'\circ \varphi$. This implies that, up to compatibility and destabilizations along the non-distinguished boundary components of $N$ and $N'$, we can obtain $\bar{\alpha}$ from $\bar{\beta}$ and from $\bar{\beta}'$. Thus $\bar{\beta}$ and $\bar{\beta}'$ are stable equivalent, consequently the induced function on the stable equivalence classes is injective. 

 \end{proof}

\begin{theorem}\label{thm:2}
  There exists a bijection between the abstract braids on $n$ strands and the virtual braids on $n$ strands. 
\end{theorem}

\begin{proof}
  We need to verify that the function induced by $G$, from $AB_n$ to $bG_n$, is well defined and that it remains injective. By abuse of notation we denote the induced map still by $G$.

Let $\bar{\beta}= (S,f,\beta,\epsilon)$ and $\bar{\beta}'= (S,f,\beta',\epsilon')$ be two abstract braid diagrams related by an $R$-move. We need to see that $G(\bar{\beta})$ is related to $G(\bar{\beta}')$ by an $\Omega 2$ or an $\Omega 3$ move. By definition of an $R$-move, there exists a neighbourhood, $D$, diffeomorphic to a disc, such that $\beta$ and $\beta'$ coincide outside $D$. Up to isotopy we can suppose that in the interval $f(D)=f'(D)$ there are no other crossings that the involved on the $R$-move. In this way to perform an $R$-move in $D$ is equivalent to perform an $\Omega 2$ or an $\Omega 3$ move in the braid Gauss diagram. Consequently  $G$ is well defined from $AB_n$ to $bG_n$. 

To prove the injectivity, let $\bar{\beta}$ and $\bar{\beta}'$ be two abstract braids diagrams such that $G(\bar{\beta})$ and $G(\bar{\beta}')$ are related by an $\Omega 2$ move. Note that the strands involved in the $\Omega 2$ move of $\beta$ (resp. of $\beta'$) in the regular neighbourhood constructed in the proof of Claim \ref{claim:thm2} look either as in the left hand side or as in the right hand side of Figure \ref{fig:thmA1} (resp. right hand side or left hand side). Deform the regular neighbourhood of the right hand side by gluing a disc in the middle, so that it looks as in the center of Figure \ref{fig:thmA1}. Then we can embed both diagrams in the same surface and relate them by a $R2$ move. Then $\bar{\beta}$ and $\bar{\beta'}$ are related by a stability and a Reidemeister move.  The case when $G(\beta)$ and $G(\beta')$ are related by an $\Omega 3$ move is proved similarly and illustrated in Figure \ref{fig:thmA2}. Thus $G$ is injective and the theorem is true. 
\end{proof}

 \begin{figure}[h]\centering
    \includegraphics[scale=0.7]{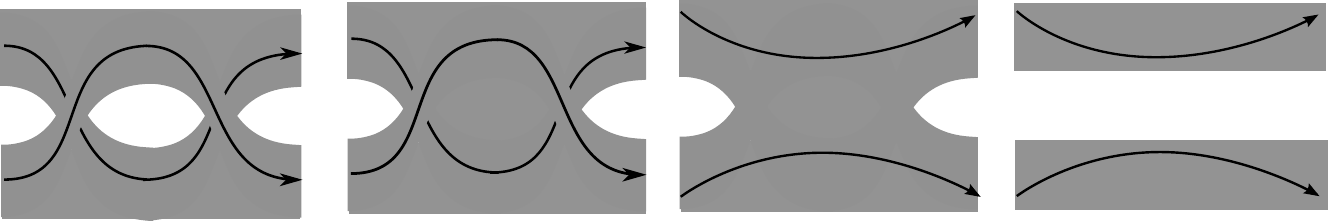} 
    \caption{Strands involved in the $\Omega 2$ move.}
    \label{fig:thmA1}
  \end{figure}    
 
 \vspace{-8pt} 
 \begin{figure}[h]\centering
    \includegraphics[scale=0.7]{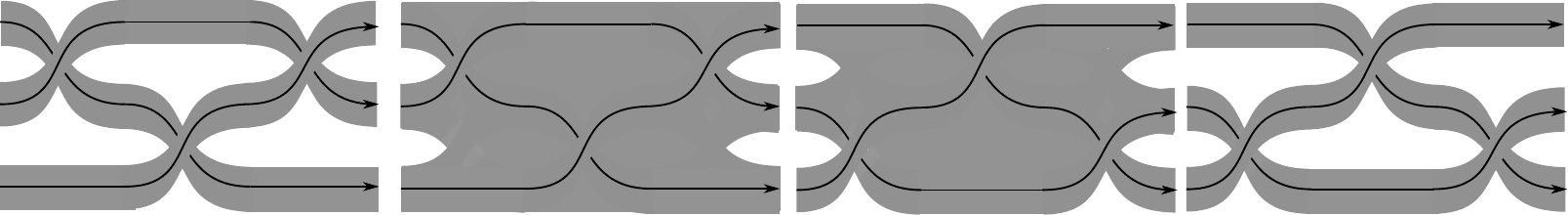} 
    \caption{Strands involved in the $\Omega 3$ move.}
    \label{fig:thmA2}
  \end{figure}    

As a consequence of the proof of the last theorem we have the next corollary.

\begin{corollary}
 Given an abstract braid diagram $\bar{\beta}$. Let $[\bar{\beta}]_s$ be its stable equivalence class. There exists a unique, up to compatibility, $\bar{\alpha}\in [\bar{\beta}]_s$, such that for all $\bar{\beta}' \in [\bar{\beta}]_s$, $\bar{\alpha}$ is obtained from $\bar{\beta}'$ by a finite number of destabilizations. 
\end{corollary}

\section{Minimal realization of an abstract braid}

  Given an abstract braid diagram $\bar{\beta}=(S,f,\beta,\epsilon)$ we call the genus of $\bar{\beta}$ to the genus of $S$. We denote it by $g(\bar{\beta})$.   
  
  Recall that $ABD_n$ denotes the set of equivalence classes of abstract braid diagrams, identified up to isotopy and compatibility equivalence. Note that the genus of an abstract braid diagram is preserved by the isotopy and compatibility equivalence. Thus we can define the genus of an element of $ABD_n$.  From now on we will confuse an abstract braid diagram with its compatibility and isotopy equivalence class. 
  
  On the other hand $AB_n$ denotes the set of stability and Reidemeister equivalence classes of abstract braid diagrams. The Reidemeister equivalence preserves the genus of an abstract braid. Denote by $TAB_n$ the set of isotopy, compatibility and Reidemeister equivalence classes of abstract braid diagrams. 
  
  Denote by $[\bar{\beta}]$ the stability and Reidemeister equivalence class of the abstract braid diagram $\bar{\beta}$. Given $[\bar{\beta}]\in AB_n$ the stability equivalence defines an order on $[\bar{\beta}]$ given by $\bar{\beta} < \bar{\beta}'$ if $\bar{\beta}$ is obtained from $\bar{\beta}'$ through Reidemeister and destability moves.  Note that a destabilization always reduces the genus of an abstract braid diagram and the genus is a non negative number. 
  
   The aim of this section is to prove that two minimal elements in $[\bar{\beta}]\in AB_n$ are related by a finite number of isotopies, compatibilities, and Reidemeister moves, that is, they represent the same element in $TAB_n$. 
   
  Recall that there is a bijective correspondence between $AB_n$ and $VB_n$ (Theorem \ref{thm:2}). In particular, for a virtual braid $\beta$ there exists a distinguished  topological representative of $\beta$, given by its minimal representative $\bar{\beta}\in TAB_n$. 
  
  Another straightforward consequence is that we can define the genus of a virtual braid as the genus of the minimal topological representative of $\beta$, and this is an invariant of the virtual braid, i.e. its value does not change up to isotopy and virtual, Reidemeister and mixed moves. 
  
  A regular braid is a virtual braid that has only regular crossings. A corollary of the previous discussion is that if a virtual braid can be reduced to a regular braid, then necessarily its genus must be zero. Eventhough, there are some virtual braids whose genus is zero and that are not regular, for example consider the virtual braid $\beta = \sigma_1 \tau_1$,  we have that $g(\beta)=0$, but it is not a regular braid (Figure \ref{fig:noRegular}).

  \begin{figure}[h]\centering
    \includegraphics[scale=0.55]{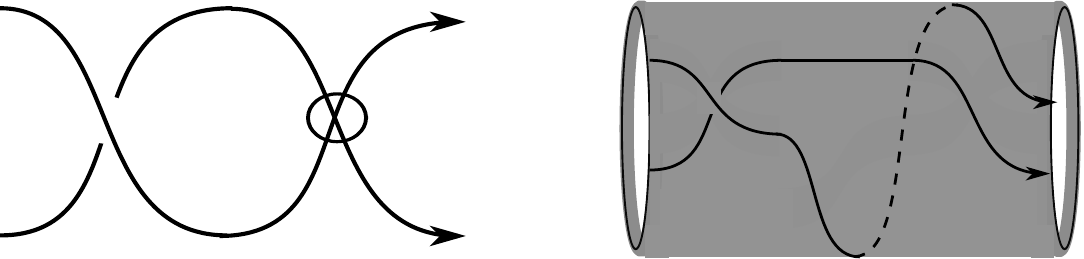} 
    \caption{No regular braid with genus $0$.}
    \label{fig:noRegular}
  \end{figure} 

   Regular braid diagrams are projections of geometric braids in $\D \times I$ on $\D$. Is well known that regular braids coincide with isotopy classes of geometric braids identified up to isotopy. In order to have a similar result for abstract braids, we need to define a geometric object in a three dimensional space, such that when it is projected on a two dimensional space we recover the Abstract braid diagrams.

\begin{definition}
	A {\it braid in a thickened surface on $n$ strands} is a triple, $\bar{\beta}=(M_S,F,\beta)$, such that: 
	\begin{enumerate}
		\item There exists a compact, connected and oriented surface $S$, such that $M_S = S\times I$.
		\item The boundary of $S$ has only two connected components, i.e. $\partial S = C_0 \sqcup C_1$, with $C_0 \approx S^{1} \approx C_1$, called \textit{distinguished boundary components}. 
		\item Each boundary component of $S$ has $n$ marked points, say $K_0 = \{a_1,\dots,a_n \} \subset C_0$ and $K_1 = \{b_1,\dots,b_n \} \subset C_1$. Such that: 
			\begin{enumerate}
				\item The elements of $K_0$ and $K_1$ are lineary ordered.
				\item Let $\kappa_0: S^{1} \rightarrow C_0$ and $\kappa_1 : S^{1} \rightarrow C_1$ be parametrizations of $C_0$ and $C_1$ compatible with the orientation of $S$. Up to isotopy we can put $a_k= \kappa_0(e^{\frac{2\pi i}{k}})$ and $b_k = \kappa_1(e^{-\frac{2\pi i}{k}})$ for $k\in \{1, \dots, n\}$.  
			\end{enumerate} 
		\item $F$ is a smooth function, $F: M_S\rightarrow I$ such that, for $i=0,1$  $$F^{-1}(\{ i \}) = C_i \times I.$$
		\item $\beta$ is an $n$-tuple of curves $\beta = (\beta_1,\dots,\beta_n)$ with:
		   \begin{enumerate}
		     \item For $k=1,\dots, n$, $\beta_k : I \rightarrow M_S$.
		     \item For $k=1,\dots, n$, $\beta_k(0)=(a_k,\frac{1}{2})$.
		     \item There exists $\sigma \in S_n$ such that for $k=1,\dots, n$, $$\beta_k(1)= (b_{\sigma(k)},\frac{1}{2}).$$
		     \item For $k=1,\dots, n$ and $t\in I$, $F\circ \beta_k(t)= t$.
		     \item For $i\neq j$, $\beta_i \cap \beta_j = \emptyset$.
		   \end{enumerate}
	\end{enumerate}
\end{definition}

From now on we fix $n\in \N$ and we say braids in a thickened surface instead of braids in a thickened surface on $n$ strands.

\begin{definition}
  An {\it isotopy} of braids in a thickened surface is a family of braids in a thickened surface $G = \{ \bar{\beta}^{s}= (M_S,F^{s},\beta^{s}) \}_{s\in I}$, such that:
  \begin{enumerate}
    \item For all $s\in I$, $K_0^s = K_0^0$ and $K_1^s = K_1^0$.
    \item For $k=1,\dots, n$, $H_k$ is continuous, where:
     \[ \begin{split}
         H_k : & I \times I \rightarrow M_S \\
               & (s,t) \mapsto \beta_k^{s}(t).
         \end{split}\]
    \item The function $H$ is smooth, where:
     \[ \begin{split}
         H : & I \times M_S \rightarrow I \\
               & (s,x) \mapsto F^{s}(x).
         \end{split}\]
  \end{enumerate}
  We say that $\bar{\beta}^{0}$ and $\bar{\beta}^{1}$ are isotopic and we denote it by $G:\bar{\beta}^{0}\simeq \bar{\beta}^{1}$. 
\end{definition}

\begin{definition}
 Given two thickened braid diagrams $\bar{\beta} = (M_S,F,\beta)$ and $\bar{\beta}' = (M_{S'},F',\beta')$, we say that they are \textit{compatible} if there exists a diffeomorphism $\varphi:M_S \rightarrow M_{S'}$ such that $F = F'\circ \varphi$ and $\beta' = \varphi \circ \beta$. We denote it by $\bar{\beta} \approx \bar{\beta}'$. Note that the compatibility relation is an equivalence relation. 
\end{definition}

  Fix a thickened surface $M_S$. Given an isotopy between two braids in $M_S$, we can decompose the isotopy in a sequence of isotopies so that, in each step, only one strand moves and a bigon is formed by the initial and terminal positions of that strand. Since the bigon is contained in a disc, the projection of this move on the surface looks like Figure \ref{fig:DeltaM}.
	
  \begin{figure}[h]\centering
    \includegraphics[scale=0.8]{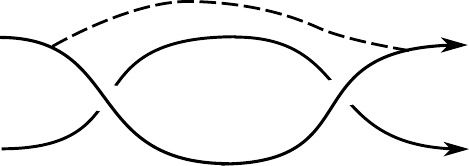} 
    \caption{$\Delta$-move.}
    \label{fig:DeltaM}
  \end{figure} 

  Such moves are called $\Delta$-moves and generate the $\Delta$-equivalence of abstract braid diagrams on $S$. Thus, there is a bijective correspondence between the $\Delta$-classes of abstract braid diagrams in $S$ and the isotopy classes of braids in $M_S$.   
  
  On the other side, the $\Delta$-equivalence generates the Reidemeister moves $R2a$, $R2b$ and $R3$ and viceversa, a $\Delta$-move can be expressed as a finite sequence of Reidemeister moves \cite[pp. 19-24]{MK99}. Consequently we have the next lemma.  

\begin{lemma}\label{lema:ReqTCB}
 There is a bijective correspondence between isotopy and compatibility classes of braids in thickened surfaces and $TAB_n$. We call the elements of $TAB_n$, \textit{thickened abstract braids} (on $n$ strands).
\end{lemma}

%
 
 From now on we will think the elements of $TAB_n$ as isotopy classes of thickened abstract braids.

\begin{definition}
  Let $\bar{\beta}= (M_S,F,\beta)\in TAB_n$. Given $A,B\subset M_S$ we say that \textit{$A$ is isotopic to $B$ relative to $\partial M_S$} if there exists a continuous function $H:A \times I \rightarrow M_S$ such that:
  \begin{enumerate}
    \item $H_0 = id_{A}$ and $H_1 (A) = B$.
    \item For all $s\in I$, $H_s$ is an embedding.
    \item For all $s\in I$, $H_s(A\cap \partial M_S) \subset \partial M_S$.
  \end{enumerate} 
  In particular $A$ is diffeomorphic to $B$, and $H$ induces an isotopy of $A\cap \partial M_S$  and $B\cap \partial M_S$ in $\partial M_S$. 
\end{definition}
  
\begin{definition}
  Given $\bar{\beta}=(M_S,F,\beta)\in TAB_n$. 
  \begin{enumerate}
    \item A {\it vertical annulus in $\bar{\beta}$} is an annulus $A \subset M_S \setminus \beta$, such that $A=C\times I \subset S\times I$ with $C$ a simple closed curve in $S$. 
    \item A {\it destabilization} of $\bar{\beta}$ is an annulus $A \subset M_S \setminus \beta$ isotopic to a vertical annulus $C\times I$ relative to $\partial M_S$, with $C$ essential and non-separating in $S$. 
    \item A {\it destabilization move} on $\bar{\beta}$ along a destabilization $A$, is to cut $M_S$ along $A$, cap the two boundary components with two thickened discs and extend the function to the obtained manifold. We also say to {\it destabilize} $\bar{\beta}$ along $A$ and we denote the obtained thickened abstract braid by $\bar{\beta}_A$. 
    \item The equivalence relation generated by these moves in the set of thickened abstract braids is called {\it stable equivalence}.   
  \end{enumerate}  

\end{definition}

As a consequence of Lemma \ref{lema:ReqTCB}, the definition of destabilization of a braid in a thickened surface is equivalent to the destabilization of an abstract braid diagram identified up to Reidemeister, isotopy and compatibility equivalence. Consequently we obtain the next proposition. 

\begin{proposition}
	The abstract braids are in bijective correspondence with the braids in thickened surfaces identified up to stable equivalence.  
\end{proposition}

Recall that the stability equivalence induces an order in $TAB_n$. This order is generated by destabilizations, i.e. given $\bar{\beta}$ and $\bar{\beta}'$, if there exists a destabilization, $A$, of $\bar{\beta}'$,  such that $\bar{\beta} \approx \bar{\beta}'_A$, then $\bar{\beta} < \bar{\beta}'$.

\begin{definition}
  Given $\bar{\beta}\in TAB_n$, a \textit{descendent} of $\bar{\beta}$ is a thickened abstract braid $\bar{\beta}'$ such that $\bar{\beta}' < \bar{\beta}$.  An \textit{irreducible descendent} of $\bar{\beta}$ is a descendent of $\bar{\beta}$ that does not admit any destabilization. 
\end{definition}


%

   Given $\bar{\beta}\in TAB_n$. Let $A \subset M_S \setminus \beta$ be an annulus isotopic to a vertical annulus $A'=C\times I$ relative to $\partial M_S$. If $C$ is not essential, we say that $A$ is {\it not essential}. Suppose $A=C\times I$ is vertical and not essential, hence $C$ bounds a disc in $S$. Let $D_0$ be the disc bounded by $C\times \{0\}$ in $S\times \{0\}$, and $D_1$ be the disc bounded by $C\times \{1\}$ in $S\times \{1\}$. Then $A\cup D_0 \cup D_1$ is homeomorphic to a sphere that bounds a ball in $M_S \setminus \beta$. To express this we say that $A$ bounds a ball, and we refer to such ball as the ball bounded by $A$. 

\begin{theorem}\label{thm:3}
	Given $[\bar{\beta}]\in AB_n$ there exists a unique irreducible descedent of $\bar{\beta}$ in $TAB_n$.
\end{theorem}

\begin{proof}
	Let $[\bar{\beta}]\in AB_n$. Suppose that $[\bar{\beta}]$ has two irreducible descendents. In this case $[\bar{\beta}]$ has a representative $\bar{\beta}=(M_S,F,\beta)$, such that $S$ is of minimal genus among the representatives of $[\bar{\beta}]$ admitting two different irreducible descendents. 
	
	Since each destabilization reduces the genus, by minimality of the genus of $S$ each destabilization of $\bar{\beta}$ has a unique irreducible descendent. Two destabilizations of $\bar{\beta}$ are called {\it descendent equivalent} if they have the same irreducible descendent. 
	
	We claim that all destabilizations in $\bar{\beta}$ are descendent equivalent. 
	 Suppose there exist two destabilizations $A_1$ and $A_2$ of $\bar{\beta}$ descendent inequivalent. 
	
\begin{claim}\label{claim:NoEmpty}
	The intersection of $A_1$ and $A_2$ is nonempty.
\end{claim}
	
	\begin{proof}
		Suppose $A_1$ and $A_2$ are disjoint. We can destabilize $\bar{\beta}$ along $A_1$ and then along $A_2$ and vice-versa. In both cases we obtain a common descendent, i.e. $(\bar{\beta}_{A_1})_{A_2}\approx (\bar{\beta}_{A_2})_{A_1}$. This is a contradiction. 
	\end{proof}	
	
	Therefore, we can suppose $A_1$ and $A_2$ intersect transversally and so that the number of curves in the intersection ($m_{1,2}\geq 1$) is minimal. Furthermore, we can choose $A_1$ and $A_2$ so that $m_{1,2}$ is minimal among inequivalent pairs of destabilizations of $\bar{\beta}$.
	
	The intersection between two transversal surfaces is a disjoint union of $1$-manifolds. A curve in $A_1\cap A_2$ is thus either a circle or an arc. A {\it horizontal circle} in an annulus $A$ is a circle that does not bound a disc in $A$ (Figure \ref{fig:verHor}). A {\it vertical arc} in an annulus $A$ is a simple arc in $A$ such that its extremes connect the two boundary components of $A$ (Figure \ref{fig:verHor}). 
	
	Given a horizontal circle $C$ in an annulus $A$, it divides $A$ in two annuli $A'$  and $A''$  (Figure \ref{fig:verHor}) such that: $$\partial A' = (\partial A \cap (S\times \{0\}) \cup C \qquad \text{ and } \qquad \partial A'' = (\partial A \cap (S\times \{1\}) \cup C.$$
 \begin{figure}[h]\centering
    \includegraphics[scale=0.35]{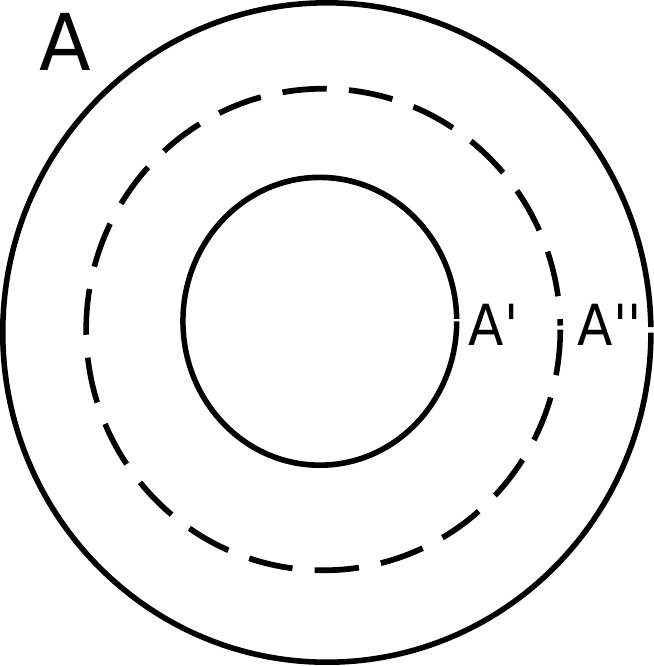} \qquad  \includegraphics[scale=0.35]{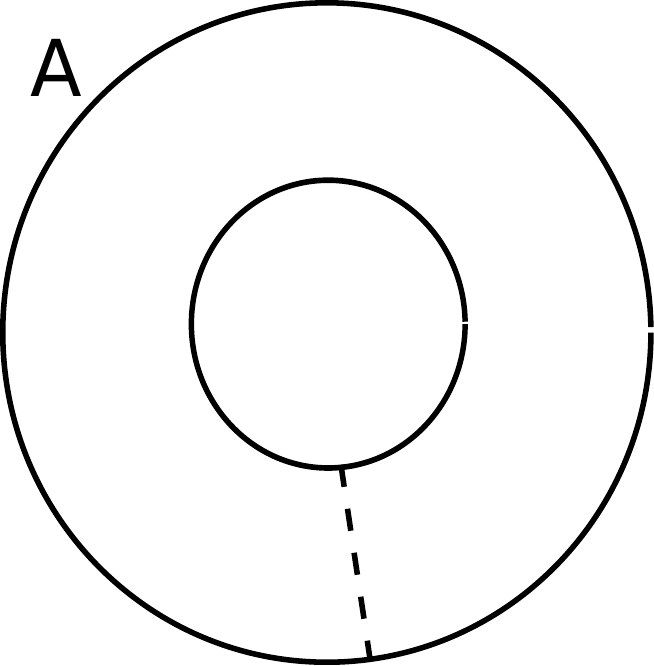}
    \caption{Horizontal circle and vertical arc in $A$.}
    \label{fig:verHor}
  \end{figure}
	
\begin{claim}\label{claim:HorVer}
	All the 1-manifolds in $A_1\cap A_2$ are either horizontal circles or vertical arcs in $A_1$ and in $A_2$.
\end{claim}	
	
	\begin{proof}
	 Suppose there exists $C\subset A_1 \cap A_2$ such that  $C$ is a non-horizontal circle in $A_1$. Thus, the circle $C$ bounds a disc $D$ in $A_1$, in particular it is null-homotopic in $M_S\setminus \beta$. On the other hand if $C$ is horizontal in $A_2$ it is homotopic to an essential circle in $S$ and so it is not null-homotopic in $M_S$. Therefore $C$ is non-horizontal in $A_2$. 
	
	Suppose that $C$ is innermost (i.e. $int(D)\cap A_2 = \emptyset$). Consider a regular neighbourhood of $D$ in $M_S\setminus \beta$, $N(D)$. The boundary of $N(D)$, $\partial N(D)$, intersects $A_2$ in two disjoint circles $C'$ and $C''$. The circle $C'$ (resp. $C''$) bounds a disc $D'$ (resp. $D''$) in $\partial N(D)$ (Figure \ref{fig:claim2}). The surface $A_2\setminus  N(D)$ has two connected components that we can complete with $D'$ and $D''$ in order to obtain two surfaces say $A_2'$ and $A_2''$. They can be spheres, annuli or discs in $M_S$. 
	
	Since $C$ is non-horizontal in $A_2$ and $C$ is innermost in $A_1$, necessarily, up to exchanging $A_2'$ with $A_2''$, $A_2'$ is a sphere and $A_2''$ is an annulus isotopic to $A_2$ (Figure \ref{fig:claim2}). By construction $A_1\cap A_2''$ has less connected components than $A_1\cap A_2$. This is a contradiction. We conclude that all the circles in $A_1 \cap A_2$ are horizontal in $A_i$ for $i=1,2$.
  \begin{figure}[h]\centering
    \includegraphics[scale=0.5]{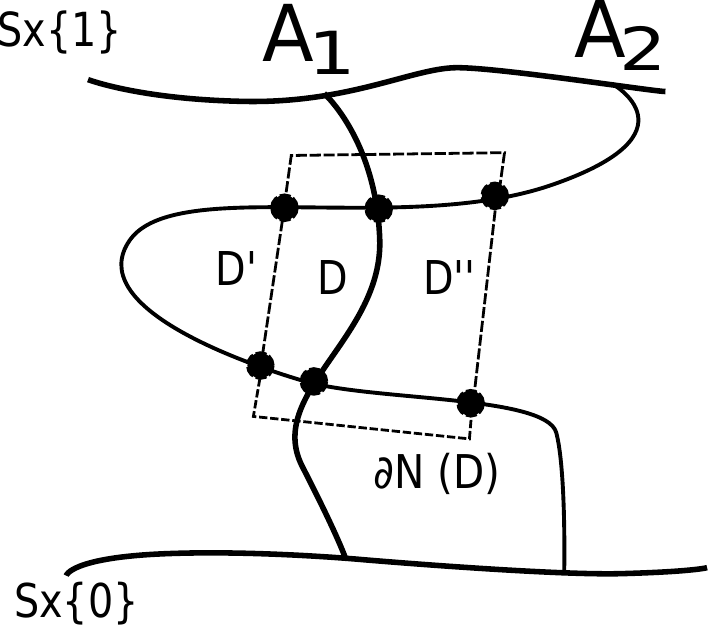} \qquad     \includegraphics[scale=0.5]{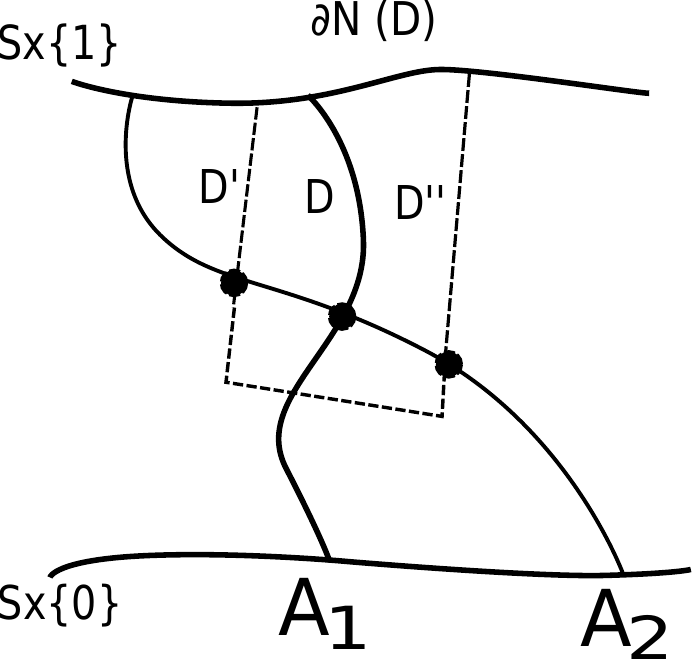}
    \caption{A non-horizontal circle and a non-vertical arc in $A_1$.}
    \label{fig:claim2}
  \end{figure}	
	
  Let $C\subset A_1\cap A_2$ be a non-vertical arc in $A_1$. Hence, the extremes of $C$ are in the same component of $\partial A_1$. Let $\alpha$ be the segment of the component of $\partial A_1$ that joins the extremes of $C$ so that $C\cup \alpha$ is a simple closed curve that bounds a disc $D$ in $A_1$. In particular $C$ is null-homotopic in $M_S$ relative to $\partial M_S$, consequently, $C$ is also a non vertical arc in $A_2$. 

	Suppose that $C$ is innermost, in the sense that $A_2 \cap int(D) = \emptyset$. Let $N(D)$ be a regular neighbourhood of $D$ in $M_S \setminus \beta$. The boundary of $N(D)$, $\partial N(D)$, intersects $A_2$ in two disjoint non-vertical arcs, $C'$ and $C''$. With a similar construction as for $C$, we can find arcs $\alpha '$ and $\alpha ''$ in $\partial N(D) \cap \partial M_S$ such that $C'\cup \alpha '$ (resp. $C''\cup \alpha ''$) bounds a disc $D'$ (resp. $D''$) in $\partial N(D)$. The surface $A_2\setminus N(D)$ has two connected components that we can complete with $D'$ and $D''$ in order to obtain two surfaces $A_2'$ and $A_2''$.  
	
	Since $C$ is non-vertical in $A_2$ and $C$ is innermost in $A_1$, necessarily, up to exchanging $A_2'$ with $A_2''$, $A_2'$ is a disc and $A_2''$ is an annulus isotopic to $A_2$ (Figure \ref{fig:claim2}). By construction $A_1\cap A_2''$ has less connected components than $A_1\cap A_2$ which is a contradiction. We conclude that all the arcs in $A_1 \cap A_2$ are vertical in $A_i$ for $i=1,2$.		
	\end{proof}		
	
\begin{claim}\label{claim:noHorC}
	The intersection $A_1\cap A_2$ does not contain any horizontal circle. 
\end{claim}
	
\begin{proof}
	Let $C\subset A_1\cap A_2$ be a horizontal circle in $A_1$. We have seen that necessarily it is a horizontal circle in $A_2$.  Then $C$ splits $A_1$ and $A_2$ in four annuli, $A_1'$, $A_1''$, $A_2'$ and $A_2''$. We can choose $C$ so that it is exterior in $A_1$ in the sense that $int(A_1'')\cap A_2 = \emptyset$. In this case the annulus $A_1''$ is isotopic to $A_2''$ in $M_S\setminus \beta$ relative to $\partial M_S$. Let $A_3$ be the annulus $A_1'' \cup A_2'$ deformed by an isotopy in such a way that it is in general position with respect to $A_1$ (Figure \ref{fig:horCir}).  
	\begin{figure}[h]\centering
    \includegraphics[scale=0.5]{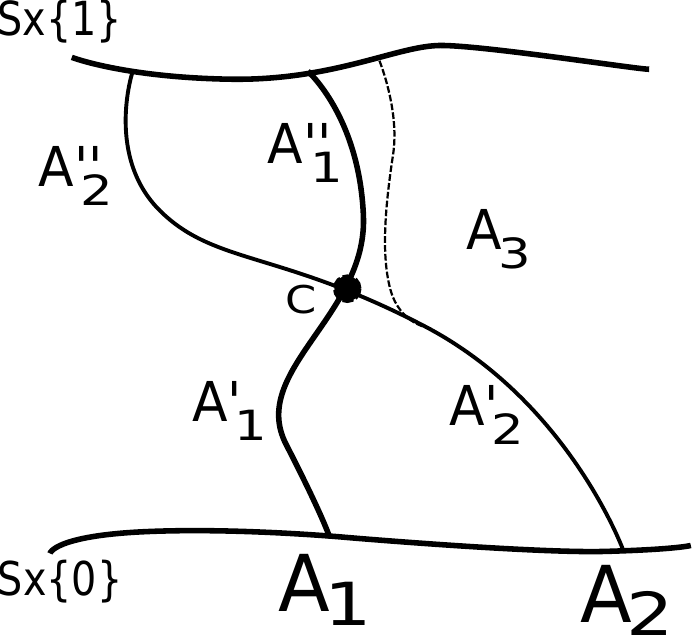}
    \caption{The intersection of two destabilizations along a horizontal circle.}
    \label{fig:horCir}
  \end{figure}		
	
	The number of curves in $A_3\cap A_1$ is strictly less than the number of curves in $A_2\cap A_1$.  Furthermore $A_2$ is isotopy equivalent to $A_1''\cup A_2'$ which is isotopy equivalent to $A_3$ by construction. Hence $A_3$ is a destabilization equivalent to $A_2$, and $A_3\cap A_1$ has strictly less curves than $A_2\cap A_1$. This is a contradiction. 
\end{proof}

\begin{claim}\label{claim:noVer}
	The intersection $A_1\cap A_2$ does not contain any vertical arc. 
\end{claim}

\begin{proof}
 Let $N$ be a regular neighbourhood of $A_1\cup A_2$ in $M_S \setminus \beta$. Then $\partial N$ is a disjoint union of $m$ surfaces in $M_S$. Since there are only vertical arcs in $A_1\cap A_2$  these surfaces are isotopic to vertical annuli, say $\partial N = B_1\sqcup B_2\sqcup \dots \sqcup B_m$. Therefore, either there is a destabilization in $\partial N$ or all the vertical annuli are non-essential.
 
 Suppose that for some $k\in \{1,\dots,m\}$, $B_k$ is a destabilization, i.e. isotopic to an essential vertical annulus. Since $B_k$ is disjoint from $A_1$ and $A_2$, it is descendent equivalent to both. This is a contradiction.

 Suppose that for all $k=1,\dots, m$, $B_k$ is isotopic to a non-essential vertical annulus. Let $E_k$ be the ball bounded by $B_k$ and $S_k = \partial E_k$. 

 We claim that there exists $k\in\{1,\dots,m\}$ such that $A_1\cup A_2 \subset E_k$. This is equivalent to say that there exists $k\in\{1,\dots,m\}$ such that $(A_1\cup A_2) \cap E_k \neq \emptyset$. It is clear that if $A_1\cup A_2 \subset E_k$ then the intersection is nonempty.  On the other hand, suppose there exists $k\in\{1,\dots,m\}$, such that $(A_1\cup A_2) \cap E_k \neq \emptyset$. Since $B_k\cap (A_1\cup A_2)=\emptyset$ and by connectivity of $A_1\cup A_2$ and of $E_k$, we have $A_1\cup A_2 \subset E_k$. 
 
 Now, suppose there exist $j,k\in \{1,\dots,m\}$, such that $j\neq k$ and $S_k\cap S_j\neq \emptyset$. Then, up to exchanging $E_k$ with $E_j$, $E_j \subset E_k$. Note that $B_j$ (resp. $B_k$) separates $M_S$ in two connected components. Furthermore, $B_k$ and $A_1\cup A_2$ (resp. $B_j$ and $A_1\cup A_2$) are in the same connected component of $M_S \setminus B_j$ (resp. $M_S \setminus B_k$). Thus $A_1\cup A_2$ is in the shell bounded by $S_j$ and $S_k$. In particular $A_1\cup A_2 \subset E_k$.

 
 If $S_i\cap S_j = \emptyset$, then $E_i\cap E_j = \emptyset$. Suppose that $E_i\cap E_j \neq \emptyset$. As $B_i \cap B_j = \emptyset$, up to exchanging $E_i$ with $E_j$, $E_i \subset E_j$ and $(S_i\cap \partial M_S) \subset (S_j \cap \partial M_S)$. This is a contradiction. 
 

 Suppose that for all $k=1,\dots,m$, $(A_1\cup A_2) \cap E_k = \emptyset$ and that $S_i\cap S_j= \emptyset$ for $i\neq j$. As $E_i\cap E_j = \emptyset$ for $i\neq j$, the connected components of $M_S \setminus (\cup_{k=1}^m B_k)$ are $int(E_1)$, \dots, $int(E_m)$, and $M_S\setminus (\cup_{k=1}^m E_k$). But $(A_1\cup A_2) \cap E_k = \emptyset$ for $k=1,\dots, m$. Thus $\beta$ and $A_1\cup A_2$ are in the same connected component. This is a contradiction, because $\partial N$ separates $\beta$ and $A_1 \cup A_2$.  We conclude that there exists $k\in\{1,\dots,m\}$ such that $A_1\cup A_2 \subset E_k$.

 For $j=1,2$ and $i=0,1$, set $\gamma_j^i = (S\times \{i\})\cap A_j$. Since $A_1\cup A_2 \subset E_k$, we have $\gamma_j^i \subset E_k$, thus $\gamma_j^i$ is null-homotopic. This is a contradiction. We conclude that there are no vertical arcs in $A_1\cap A_2$.	
\end{proof}

Finally by Claim \ref{claim:NoEmpty}, $A_1\cap A_2 \neq \emptyset$. On the other hand by Claim \ref{claim:HorVer}, $A_1\cap A_2$ has only vertical arcs or horizontal circles. But Claims \ref{claim:noHorC} and \ref{claim:noVer} state that $A_1\cap A_2$ does not have neither horizontal circles nor vertical arcs, thus $A_1\cap A_2=\emptyset$. This is a contradiction. We conclude that there are no descendent inequivalent destabilizations of $\bar{\beta}$, thus there is a unique irreducible descendent. 
\end{proof}

\section*{Acknowledgments}
{\small
    I am very grateful to my Ph.D. advisor, Luis Paris, for helpful conversations, pertinent remarks about the manuscript and many ideas embedded in the article. I would also like to thank to the anonymous reviewer who called my attention on some interesting perspectives and for the remarks and style corrections on the text.  This work was funded by the National Council on Science and Technology, Mexico (CONCyT), under the graduate fellowship 214898.}

\end{document}